\documentclass[12pt,dvips]{amsart}
\usepackage{amsfonts, amssymb, latexsym, graphics, color}

\theoremstyle{plain}
\newtheorem{theorem}{Theorem}[section]
\newtheorem{corollary}[theorem]{Corollary}
\newtheorem{lemma}[theorem]{Lemma}
\newtheorem*{lemma:covexonecomp}{Lemma \ref{lem:covexonecomp}}
\newtheorem*{lemma:covexfibub}{Lemma \ref{lem:covexfibub}}
\newtheorem*{lemma:covexfiblb}{Lemma \ref{lem:covexfiblb}}
\newtheorem*{lemma:ressm}{Lemma \ref{lem:ressm}}
\newtheorem*{lemma:fibgeom}{Lemma \ref{lem:fibgeom}}
\newtheorem*{lemma:fiblb}{Lemma \ref{lem:fiblb}}
\newtheorem*{lemma:fibub}{Lemma \ref{lem:fibub}}
\newtheorem*{lemma:excgeom}{Lemma \ref{lem:excgeom}}
\newtheorem*{lemma:exccomb}{Lemma \ref{lem:exccomb}}

\theoremstyle{definition}

\newtheorem{conjecture}[theorem]{Conjecture}

\newtheorem{remark}[theorem]{Remark}

\numberwithin{equation}{section}

\newcommand{\Sym}{\ensuremath{S}}
\newcommand{\KL}[1]{\ensuremath{KL_{#1}}}

\newcommand{\MS}[1]{\ensuremath{MS_{#1}}}

\DeclareMathOperator{\fl}{fl}
\DeclareMathOperator{\ms}{maxsing}
\begin{document}

\title{Permutations with Kazhdan-Lusztig polynomial $P_{id,w}(q)=1+q^h$}
\author[Alexander Woo]{Alexander Woo \\
(Appendix by Sara Billey and Jonathan Weed)}
\address{Mathematics, Statistics, and Computer Science, Saint Olaf College, 1520 Saint Olaf Ave., Northfield, MN 55057}
\email{woo@stolaf.edu}

\thanks{AW gratefully acknowledges support from NSF VIGRE grant DMS-0135345.
SB gratefully acknowledges support from NSF grant DMS-0800978.
JW gratefully acknowledges support from NSF REU grant DMS-0754486.}

\date{\today}
\begin{abstract}

Using resolutions of singularities introduced by Cortez and a method
for calculating Kazhdan-Lusztig polynomials due to Polo, we prove the
conjecture of Billey and Braden characterizing permutations $w$ with
Kazhdan-Lusztig polynomial $P_{id,w}(q)=1+q^h$ for some $h$.

\end{abstract}
\maketitle

\tableofcontents

\section{Introduction}

Kazhdan-Lusztig polynomials are polynomials $P_{u,w}(q)$ in one
variable associated to each pair of elements $u$ and $w$ in the
symmetric group $S_n$ (or more generally in any Coxeter group).  They
have an elementary definition in terms of the Hecke
algebra~\cite{KazLus, Hum, BjoBre} and numerous applications in
representation theory, most notably in~\cite{KazLus, BeilBern,
BryKash}, and the geometry of homogeneous spaces~\cite{KazLusIH,
Deoratsm}.  While their definition makes it fairly easy to compute any
particular Kazhdan-Lusztig polynomial, on the whole they are poorly
understood.  General closed formulas are known~\cite{BilBre, Bre}, but
they are fairly complicated; furthermore, although they are known to
be positive (for $S_n$ and other Weyl groups), these formulas have
negative signs.  For $S_n$, positive formulas are known only for $3412$
avoiding permutations~\cite{Las, LasSchu}, $321$-hexagon avoiding
permutations~\cite{BilWar321}, and some isolated cases related to the
generic singularities of Schubert
varieties~\cite{BilWar,Man2,Cor2,War}.

One important interpretation of Kazhdan-Lusztig polynomials is as
local intersection homology Poincar\'e polynomials for Schubert
varieties.  This interpretation, originally established by Kazhdan and
Lusztig~\cite{KazLusIH}, shows, in an entirely non-constructive
manner, that Kazhdan-Lusztig polynomials have nonnegative integer
coefficients and constant term 1.  Furthermore, as shown by
Deodhar~\cite{Deoratsm}, $P_{id,w}(q)=1$ (for $S_n$) if and only if
the Schubert variety $X_w$ is smooth, and, more generally,
$P_{u,w}(q)=1$ if and only if $X_w$ is smooth over the Schubert cell
$X^\circ_u$.

The purpose of this paper is to prove the following theorem.

\begin{theorem}
\label{thm:main}
Suppose the singular locus of $X_w$ has exactly one irreducible
component, and $w$ avoids the patterns $653421$, $632541$, $463152$,
$526413$, $546213$, and $465132$.  Then $P_{id,w}(1)=2$.

More precisely, when the hypotheses are satisfied, $P_{id,w}(q)=1+q^h$
where $h$ is the minimum height of a $3412$ embedding, with $h=1$ if
no such embedding exists.
\end{theorem}

Here, a $3412$ embedding is a sequence of indices $i_1<i_2<i_3<i_4$
such that $w(i_3)<w(i_4)<w(i_1)<w(i_2)$, and its height is
$w(i_1)-w(i_4)$.  Given the first part of the theorem, the second part
can be immediately deduced from the unimodality of Kazhdan-Lusztig
polynomials~\cite{Irv,BraMac} and the calculation of the
Kazhdan-Lusztig polynomial at the unique generic
singularity~\cite{BilWar,Man2,Cor2}.  Indeed, unimodality and this
calculation imply the following corollary.

\begin{corollary}
\label{cor:main}
Suppose $w$ satisfies the hypotheses of Theorem~\ref{thm:main}.  Let
$X_v$ be the singular locus of $X_w$.  Then $P_{u,w}(q)=1+q^h$ (with
$h$ as in Theorem~\ref{thm:main}) if $u\leq v$ in Bruhat order, and
$P_{u,w}(q)=1$ otherwise.
\end{corollary}

The permutation $v$ and the singular locus in general has a
combinatorial description given in Theorem~\ref{thm:singlocus}, which
was originally proved independently in~\cite{BilWar, Cor2, KLR, Man}.

Theorem~\ref{thm:main} was conjectured by Billey and Braden
\cite{BilBra}.  They claim the converse in their paper.  An outline of
the proof is as follows.  If $P_{id,w}(1)=1$ then $X_w$ is
nonsingular~\cite{Deoratsm}.  The methods for calculating
Kazhdan-Lusztig polynomials due to Braden and MacPherson~\cite{BraMac}
show that $P_{id,w}(1)\leq 2$ implies that the singular locus of $X_w$
has at most one component.  That $P_{id,w}(1)\leq 2$ implies the
pattern avoidance conditions follows from \cite[Thm. 1]{BilBra} and
the computation of Kazhdan-Lusztig polynomials for the six pattern
permutations.

While this paper was being written, Billey and Weed found an
alternative formulation of Theorem~\ref{thm:main} purely in terms of
pattern avoidance, replacing the condition that the singular locus of
$X_w$ have only one component with sixty patterns.  They have
graciously agreed to allow their result, Theorem~\ref{thm:pats},
to be included in an appendix to this paper.
Theorem~\ref{thm:pats} also provides an alternate method for
proving the converse to Theorem~\ref{thm:main} using only
\cite[Thm. 1]{BilBra} and bypassing the methods of \cite{BraMac}.

To prove Theorem~\ref{thm:main}, we study resolutions of singularities
for Schubert varieties that were introduced by Cortez~\cite{Cor1,
Cor2} and use an interpretation of the Decomposition Theorem
\cite{BeilBernDel} given by Polo \cite{Polo} which allows computation
of Kazhdan-Lusztig polynomials $P_{v,w}$ (and more generally local
intersection homology Poincar\'e polynomials for appropriate
varieties) from information about the fibers of a resolution of
singularities.  In the $3412$-avoiding case, we use a resolution of
singularities from~\cite{Cor1} and a second resolution of
singularities which is closely related.  An alternative approach which
we do not take here would be to analyze the algorithm of
Lascoux~\cite{Las} for calculating these Kazhdan-Lusztig polynomials.
For permutations containing $3412$, we use one of the partial
resolutions introduced in~\cite{Cor2} for the purpose of determining
the singular locus of $X_w$.  Under the conditions described above,
this partial resolution is actually a resolution of singularities, and
we use Polo's methods on it.

Though we have used purely geometric arguments, it is possible to
combinatorialize the calculation of Kazhdan-Lusztig polynomials from
resolutions of singularities using a Bialynicki-Birula
decomposition~\cite{BialI,BialII,Car} of the resolution.  See
Remark~\ref{rem:comb-calc} for details.

Corollary~\ref{cor:main} suggests the problem of describing all pairs
$u$ and $w$ for which $P_{u,w}(1)=2$.  It seems possible to extend the
methods of this paper to characterize such pairs; presumably $X_u$
would need to lie in no more than one component of the singular locus
of $X_w$, and $[u,w]$ would need to avoid certain intervals (see
Section~\ref{sect:patavoid}).  Any further extension to characterize
$w$ for which $P_{id,w}(1)=3$ is likely to be extremely
combinatorially intricate.  An extension to other Weyl groups would
also be interesting, not only for its intrinsic value, but because
methods for proving such a result may suggest methods for proving any
(currently nonexistent) conjecture combinatorially describing the
singular loci of Schubert varieties for these other Weyl groups.

I wish to thank Eric Babson for encouraging conversations and Sara
Billey for helpful comments and suggestions on earlier drafts.  I used
Greg Warrington's software~\cite{WarKLPOL} for computing
Kazhdan-Lusztig polynomials in explorations leading to this work.

\section{Preliminaries}

\subsection{The symmetric group and Bruhat order}

We begin by setting notation and basic definitions.  We let $S_n$
denote the symmetric group on $n$ letters.  We let $s_i\in S_n$ denote
the adjacent transposition which switches $i$ and $i+1$; the elements
$s_i$ for $i=1,\ldots,n-1$ generate $S_n$.  Given an element $w\in
S_n$, its {\bf length}, denoted $\ell(w)$, is the minimal number of
generators such that $w$ can be written as $w=s_{i_1}s_{i_2}\cdots
s_{i_\ell}$.  An {\bf inversion} in $w$ is a pair of indices $i<j$
such that $w(i)>w(j)$.  The length of a permutation $w$ is equal to
the number of inversions it has.

Unless otherwise stated, permutations are written in one-line
notation, so that $w=3142$ is the permutation such that $w(1)=3$,
$w(2)=1$, $w(3)=4$, and $w(4)=2$.

Given a permutation $w\in S_n$, the {\bf graph} of $w$ is the set of
points $(i,w(i))$ for $i\in\{1,\ldots,n\}$.  We will draw graphs
according to the Cartesian convention, so that $(0,0)$ is at the
bottom left and $(n,0)$ the bottom right.

The {\bf rank function} $r_w$ is defined by
$$r_w(p,q)=\#\{i\mid 1\leq i\leq p, 1\leq w(i)\leq q\}$$ for any
$p,q\in \{1,\ldots,n\}$.  We can visualize $r_w(p,q)$ as the number of
points of the graph of $w$ in the rectangle defined by $(1,1)$ and
$(p,q)$.  There is a partial order on $S_n$, known as {\bf Bruhat
order}, which can be defined as the reverse of the natural partial
order on the rank function; explicitly, $u\leq w$ if $r_u(p,q)\geq
r_w(p,q)$ for all $p,q\in\{1,\ldots,n\}$.  The Bruhat order and the
length function are closely related.  If $u<w$, then
$\ell(u)<\ell(w)$; moreover, if $u<w$ and $j=\ell(w)-\ell(u)$, then
there exist (not necessarily adjacent) transpositions $t_1,\ldots,t_j$
such that $u=t_j\cdots t_1w$ and
$\ell(t_{i+1}\cdots t_1w)=\ell(t_i\cdots t_1w)-1$ for all $i$, $1\leq i<j$.

\subsection{Schubert varieties}

Now we briefly define Schubert varieties.  A {\bf (complete) flag}
$F_\bullet$ in $\mathbb{C}^n$ is a sequence of subspaces
$\{0\}\subseteq F_1\subset F_2 \subset \cdots \subset F_{n-1} \subset
F_n=\mathbb{C}^n$, with $\dim F_i = i$.  As a set, the {\bf flag
variety} $\mathcal{F}_n$ has one point for every flag in
$\mathbb{C}^n$.  The flag variety $\mathcal{F}_n$ has a geometric
structure as $GL(n)/B$, where $B$ is the group of invertible upper
triangular matrices, as follows.  Given a matrix $g\in GL(n)$, we can
associate to it the flag $F_\bullet$ with $F_i$ being the span of the
first $i$ columns of $g$.  Two matrices $g$ and $g^\prime$ represent
the same flag if and only if $g^\prime=gb$ for some $b\in B$, so
complete flags are in one-to-one correspondence with left $B$-cosets
of $GL(n)$.

Fix an ordered basis $e_1,\ldots,e_n$ for $\mathbb{C}^n$, and let
$E_\bullet$ be the flag where $E_i$ is the span of the first $i$ basis
vectors.  Given a permutation $w\in S_n$, the {\bf Schubert cell}
associated to $w$, denoted $X^\circ_w$, is the subset of
$\mathcal{F}_n$ corresponding to the set of flags
\begin{equation}
\label{eqn:schubvardef}
\{F_\bullet\mid \dim(F_p\cap E_q)= r_w(p,q) \ \forall p,q\}.
\end{equation}
The conditions in \ref{eqn:schubvardef} are called {\bf rank
conditions} The {\bf Schubert variety} $X_w$ is the closure of the
Schubert cell $X^\circ_w$; its points correspond to the flags
$$\{F_\bullet\mid \dim(F_p\cap E_q)\geq r_w(p,q) \ \forall p,q\}.$$
Bruhat order has an alternative definition in terms of Schubert
varieties; the Schubert variety $X_w$ is a union of Schubert cells,
and $u\leq w$ if and only if $X^\circ_u\subset X_w$.  In each Schubert
cell $X^\circ_w$ there is a {\bf Schubert point} $e_w$, which is the
point associated to the permutation matrix $w$; in terms of flags, the
flag $E^{(w)}_\bullet$ corresponding to $e_w$ is defined by
$E^{(w)}_i=\mathbb{C}\{e_{w(1)},\ldots,e_{w(i)}\}$.  The Schubert cell
$X^\circ_w$ is the orbit of $e_w$ under the left action of the group
$B$.

Many of the rank conditions in~(\ref{eqn:schubvardef}) are actually
redundant.  Fulton \cite{Ful} showed that for any $w$ there is a
minimal set, called the {\bf coessential set}, of rank conditions
which suffice to define $X_w$.  To be precise, the coessential set is
given by
$$\operatorname*{Coess}(w)=\{(p,q)\mid w(p)\leq q<w(p+1),
w^{-1}(q)\leq p<w^{-1}(q+1)\}
\footnote{Fulton~\cite{Ful} indexes Schubert varieties in a manner
reversed from our indexing as it is more convenient in his context.
As a result, his Schubert varieties are defined by inequalities in the
opposite direction, and he defines the {\bf essential set} with
inequalities reversed from ours.  Our conventions also differ from
those of Cortez~\cite{Cor1} in replacing her $p-1$ with $p$.},$$
and a flag $F_\bullet$ corresponds to
a point in $X_w$ if and only if $\dim(F_p\cap E_q)\geq r_w(p,q)$ for
all $(p,q)\in\operatorname*{Coess}(w)$.

While we have distinguished between points in flag and Schubert
varieties and the flags they correspond to here, we will be freely
ignore this distinction in the rest of the paper.

\subsection{Pattern avoidance and interval pattern avoidance}
\label{sect:patavoid}

Let $v\in S_m$ and $w\in S_n$, with $m\leq n$.  A {\bf (pattern)
embedding} of $v$ into $w$ is a set of indices $i_1<\cdots<i_m$ such
that the entries of $w$ in those indices are in the same relative
order as the entries of $v$.  Stated precisely, this means that, for
all $j,k\in\{1,\ldots,m\}$, $v(j)<v(k)$ if and only if
$w(i_j)<w(i_k)$.  A permutation $w$ is said to {\bf avoid} $v$ if
there are no embeddings of $v$ into $w$.

Now let $[x,v]\subseteq S_m$ and $[u,w]\subseteq S_n$ be two intervals
in Bruhat order.  An {\bf (interval) (pattern) embedding} of $[x,v]$
into $[u,w]$ is a simultaneous pattern embedding of $x$ into $u$ and
$v$ into $w$ using the same set of indices $i_1<\cdots<i_m$, with the
additional property that $[x,v]$ and $[u,w]$ are isomorphic as posets.
For the last condition, it suffices to check that
$\ell(v)-\ell(x)=\ell(w)-\ell(u)$~\cite[Lemma 2.1]{WooYong}.

Note that given the embedding indices $i_1<\cdots<i_m$, any three of
the four permutations $x$, $v$, $u$, and $w$ determine the fourth.
Therefore, for convenience, we sometimes drop $u$ from the terminology
and discuss embeddings of $[x,v]$ in $w$, with $u$ implied.  We also
say that $w$ {\bf (interval) (pattern) avoids} $[x,v]$ if there are no
interval pattern embeddings of $[x,v]$ into $[u,w]$ for any $u\leq w$.

\subsection{Singular locus of Schubert varieties}
\label{sect:singlocusdesc}

Now we describe combinatorially the singular loci of Schubert
varieties.  The results of this section are due independently to
Billey and Warrington \cite{BilWar}, Cortez \cite{Cor1, Cor2}, Kassel,
Lascoux, and Reutenauer \cite{KLR}, and Manivel \cite{Man}.

Stated in terms of interval pattern embeddings as
in~\cite[Thm. 6.1]{WooYong}, the theorem is as follows.  We use the
convention that the segment ``$j\cdots i$'' means $j, j-1, j-2,\ldots,
i+1,i$.  In particular, if $j<i$ then the segment is empty.

\begin{theorem}
\label{thm:singlocus}
The Schubert variety $X_w$ is singular at $e_{u^\prime}$ if and only if there
exists $u$ with $u^\prime\leq u<w$ such that one of the
following (infinitely many) intervals embeds in $[u,w]$:
\begin{enumerate}
\item[I:] $\big[(y+1)a\cdots 1(y+z+2)\cdots(y+2),\ \ \ (y+z+2)(y+1)y\cdots 2 (y+z+1)\cdots(y+2) 1\big]$ for some integers $y,z>0$.
\item[IIA:] $\big[(y+1)\cdots 1 (y+3) (y+2) (y+z+4)\cdots(y+4),\ \ \ (y+3)(y+1)\cdots 2(y+z+4)1(y+z+3)\cdots(y+4)(y+2)\big]$ for some integers $y,z\geq0$.
\item[IIB:] $\big[1(y+3)\cdots2(y+4), \ \ \  (y+3)(y+4)(y+2)\cdots312 \big]$ for some integer $y>1$.
\end{enumerate}

Equivalently, the irreducible components of the singular locus of
$X_w$ are the subvarieties $X_u$ for which one of these intervals
embeds in $[u,w]$.

\end{theorem}

We call irreducible components of the singular locus of $X_w$ type I
or type II (or IIA or IIB) depending on the interval which embeds in
$[u,w]$, as labelled above.

We also wish to restate this theorem in terms of the graph of $w$,
which is closer in spirit to the original statements~\cite{BilWar,
Cor2, KLR, Man}.

A type I component of the singular locus of $X_w$ is associated to an
embedding of $(y+z+2)(y+1)y\cdots 2 (y+z+1)\cdots(y+2) 1$ into $w$.
If we label the embedding by
$i=j_0<j_1<\cdots<j_y<k_1<\cdots<k_z<m=k_{z+1}$, the requirement that
these positions give the appropriate interval embedding is equivalent
to the requirement that the regions $\{(p,q)\mid j_{r-1}<p<j_r,
w(j_r)<q<w(i)\}$, $\{(p,q)\mid k_s<p<k_{s+1}, w(m)<q<w(k_s)\}$, and
$\{(p,q)\mid j_y<p<k_1,w(m)<q<w(i)\}$ contain no point $(p,w(p))$ in
the graph of $w$ for all $r$, $1\leq r\leq y$, and for all $s$, $1\leq
s\leq z$.  This is illustrated in Figure~\ref{fig:typeIsingconfig}.
We will usually say that the type I component given by this embedding
is defined by $i$, the set $\{j_1,\ldots,j_y\}$, the set
$\{k_1,\ldots,k_z\}$, and $m$.

\begin{figure}[htbp]
\begin{center}

\input{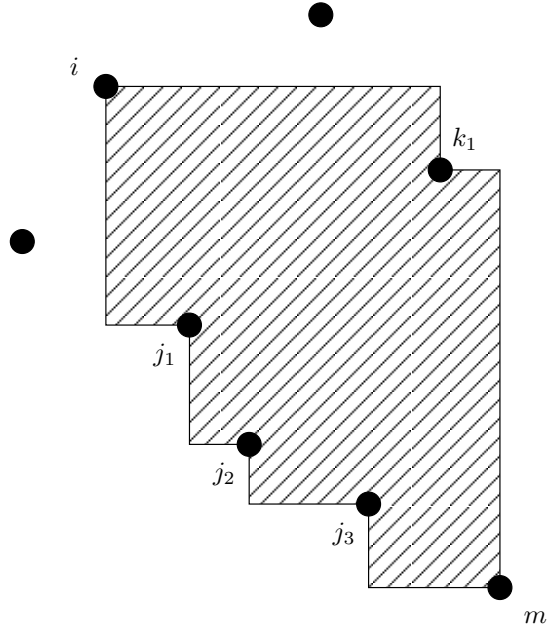}

\caption{A type I embedding with $y=3$, $z=1$, defining a component of
the singular locus for $w=685392714$.  The shaded region is not
allowed have points in the graph of $w$.}
\label{fig:typeIsingconfig}
\end{center}
\end{figure}

Every type II component of the singular locus $X_w$ is defined by four
indices $i<j<k<m$ which gives an embedding of $3412$ into $w$.  The
interval pattern embedding requirement forces the regions $\{(p,q)\mid
i<p<j, w(m)<q<w(i)\}$, $\{(p,q)\mid j<p<k, w(i)<q<w(j)\}$,
$\{(p,q)\mid k<p<m, w(m)<q<w(i)\}$, and $\{(p,q)\mid j<p<k,
w(k)<q<w(m)\}$ to have no points in the graph of $w$.  We call these
regions the {\bf critical regions} of the $3412$ embedding, and if
they are empty, we call $i<j<k<m$ a {\bf critical $3412$ embedding}
whether or not they are part of a type II component.

Given a critical $3412$ embedding $i<j<k<m$, let $B=\{p\mid j<p<k,
w(m)<w(p)<w(i)\}$, $A_1=\{p\mid i<p<j, w(k)<w(p)<w(m)\}$, $A_2=\{p\mid
k<p<m, w(i)<w(p)<w(j)\}$, and $A=A_1\cup A_2$.  We call these regions
the $A$, $A_1$, $A_2$, and $B$ regions associated to our critical
$3412$ embedding.  This is illustrated in
Figure~\ref{fig:typeIIsingconfig}.  If $w(b_1)>w(b_2)$ for all
$b_1<b_2\in B$, we say our critical $3412$ embedding is {\bf reduced}.
If a critical embedding is not reduced, there will necessarily be at
least one critical $3412$ embedding involving $i$, $j$, and two
indices in $B$, and one involving two indices in $B$, $k$, and $m$; by
induction each will include at least one reduced critical $3412$
embedding.

\begin{figure}[htbp]
\begin{center}

\input{typeIIsingconfig.pstex_t}

\caption{A critical $3412$ embedding in $w=2574136$.  The shaded
regions are the critical regions of the embedding.}
\label{fig:typeIIsingconfig}
\end{center}
\end{figure}

We associate one or two irreducible components of the singular locus
of $X_w$ to every reduced critical $3412$ embedding.  If $B$ is empty,
then the embedding is part of a component of type IIA.  If $A$ is
empty, then the embedding is part of a component of type IIB.  Note
that any type II component of the singular locus is associated to
exactly one reduced critical $3412$ embedding.  However, if both $A$
and $B$ are nonempty, then we do not have a type II component.  In
this case, we can associate a type I component of the singular locus
to our reduced critical $3412$ embedding $i<j<k<m$.  When both $A_1$
and $B$ are nonempty, then $i$, a nonempty subset of $A_1$, $B$, and
$k$ define a type I component; in this case $w$ has an embedding of
$526413$.  When both $A_2$ and $B$ are nonempty, then $j$, $B$, a
nonempty subset of $A_2$, and $m$ define a type I component; in this
case $w$ has an embedding of $463152$.  When $A_1$, $A_2$, and $B$ are
all nonempty, we have two distinct type I components associated to our
$3412$ embedding.  Note that it is possible for a type I component to
be associated to more than one reduced critical $3412$ embedding, as
in the permutation $47318625$.

\section{Necessity in the covexillary case}
\label{sect:covex}

We begin with the case where $w$ avoids $3412$; such a permutation is
commonly called {\bf covexillary}.  We show here that, if $w$ is
covexillary, the singular locus of $X_w$ has only one component, and
$w$ avoids $653421$ and $632541$, then $P_{id,w}(q)=1+q$.  Throughout
this section $w$ is assumed to be covexillary unless otherwise noted.

\subsection{The Cortez-Zelevinsky resolution}
For a covexillary permutation, the coessential set has the special
property that, for any
$(p,q),(p^\prime,q^\prime)\in\operatorname*{Coess}(w)$ with $p\leq
p^\prime$, we also have $q\leq q^\prime$.  Therefore have a natural
total order on the coessential set, and we label its elements
$(p_1,q_1),\ldots,(p_k,q_k)$ in order.  We let $r_i=r_w(p_i,q_i)$;
note that, by the definition of $r_w$ and the minimality of the
coessential set, $r_i<r_j$ when $i<j$.  When $r_i=\min\{p_i,q_i\}$, we
call $(p_i,q_i)$ an {\bf inclusion element} of the coessential set,
since the condition it implies for $X_w$ will either be
$E_{q_i}\subseteq F_{p_i}$ (if $r_i=q_i$) or $F_{p_i}\subseteq
E_{q_i}$ (if $r_i=p_i$).

Zelevinsky~\cite{Zel} described some resolutions of singularities of
$X_w$ in the case where $w$ has at most one ascent (meaning that
$w(i)<w(i+1)$ for at most one index $i$), explaining a formula of
Lascoux and Sch\"utzenberger~\cite{LasSchu} for Kazhdan-Lusztig
polynomials $P_{v,w}(q)$ in that case.  Following a generalization by
Lascoux~\cite{Las} of this formula to covexillary permutations,
Cortez~\cite{Cor1} generalized the Zelevinsky resolution to this case.

Let $\mathcal{F}_{i_1,\ldots,i_k}$ denote the partial flag manifold
whose points correspond to flags whose component subspaces have
dimensions $i_1<\cdots<i_k$.  Define the configuration variety $Z_w$
by $$Z_w:=\{(G_\bullet,F_\bullet)\in
\mathcal{F}_{r_1,\ldots,r_k}(\mathbb{C}^n) \times X_w \mid
G_{r_i}\subseteq (F_{p_i}\cap E_{q_i}) \ \forall i\}.$$ Cortez shows
that the projection $\pi_2: Z_w\rightarrow X_w$ is a resolution of
singularities.  She furthermore shows that the exceptional locus of
$\pi_2$ is precisely the singular locus of $X_w$, and describes a
one-to-one correspondence between components of the singular locus of
$X_w$ and elements of the coessential set which are not inclusion
elements.  (This last fact about the singular locus was implicit in
Lascoux's formula \cite{Las} for covexillary Kazhdan-Lusztig
polynomials.)

We now have the following lemma, whose proof is deferred to
Section~\ref{sect:lemmas}.

\begin{lemma}
\label{lem:covexonecomp}
Suppose the singular locus of $X_w$ has only one component.  If $w$
contains both $53241$ and $52431$, then $w$ contains $632541$.
\end{lemma}

This lemma allows us to treat separately the two cases where $w$
avoids $53241$ and where $w$ avoids $52431$.  We treat first the case
where $w$ avoids $53241$, for which we use the resolution of
singularities just described.  The case where $w$ avoids $52431$
requires the use of a resolution of singularities which is dual (in
the sense of dual vector spaces) to the one just described; we will
describe this resolution at the end of this section.

\subsection{The $53241$-avoiding case}
\label{sect:53241-avoid}

In this subsection we show that $P_{id,w}(q)=1+q$ when the singular
locus of $X_w$ has exactly one component and $w$ avoids $653421$ and
$53241$.  To maintain the flow of the argument, proofs of lemmas are
deferred to Section~\ref{sect:lemmas}.

When $(p_j,q_j)$ is an inclusion element, then $\dim(F_{p_j}\cap
E_{q_j})=r_j$ for any flag $F_\bullet$ in $X_w$ and not merely generic
flags in $X_w$.  Therefore, given any $F_\bullet$ we will have only
one choice for $G_{r_j}$, namely $F_{p_j}\cap E_{q_j}$, in the fiber
$\pi_2^{-1}(F_\bullet)$.  In particular, for the flag $E_\bullet$, any
$G_\bullet$ in the fiber $\pi^{-1}(E_\bullet)$ will have
$G_{r_j}=E_{r_j}$.  Now let $i$ be the unique index such that
$(p_i,q_i)$ is not an inclusion element; there is only one such index
since the singular locus of $X_w$ has only one irreducible component.
For convenience, we let $p=p_i$, $q=q_i$, and $r=r_i$.  Now we have
the following lemmas.  (In the case where $i=1$, we define
$p_0=q_0=r_0$.)

\begin{lemma}
\label{lem:covexfibub}
Suppose $w$ avoids $653421$ (and $3412$).  Then $\min\{p,q\}=r+1$.
\end{lemma}

\begin{lemma}
\label{lem:covexfiblb}
Suppose $w$ avoids $53241$ (and $3412$).  Then $r_{i-1}=r-1$.
\end{lemma}

By definition, $G_r\supseteq G_{r_{i-1}}$.  Therefore, the fiber
$\pi_2^{-1}(e_{id})=\pi_2^{-1}(E_\bullet)$ is precisely
$$\{(G_\bullet,E_\bullet)\mid G_{r_j}=E_{r_j}\ \mathrm{for}\ j\neq i\
\mathrm{and}\ E_{r-1}=E_{r_{i-1}}\subseteq G_{r}\subseteq(E_p\cap E_q)=E_{r+1}\}.$$
This fiber is clearly isomorphic to $\mathbb{P}^1$.

By Polo's interpretation \cite{Polo} of the Decomposition
Theorem \cite{BeilBernDel},
$$H_{z,\pi_2}(q)=P_{z,w}(q)+\sum_{z\leq v<w}
q^{\ell(w)-\ell(v)}E_v(q)P_{z,v}(q),$$ where
$$H_{z,\pi_2}(q)=\sum_{i\geq0} q^i\dim H^{2i}(\pi_2^{-1}(e_z)),$$ and
the $E_v(q)$ are some Laurent polynomials in $q^{\frac{1}{2}}$,
depending only on $v$ and $\pi_2$ and not on $z$, which have with
positive integer coefficients and satisfy the identity
$E_v(q)=E_v(q^{-1})$.  Since the fiber of $\pi_2$ at $e_{id}$ is
$\mathbb{P}^1$, it follows that $H_{id,\pi_2}(q)=1+q$.  As
$P_{id,w}(q)\neq 1$ (since by assumption $X_w$ is singular), and all
coefficients of all polynomials involved must be nonnegative integers,
$E_v(q)=0$ for all $v$ and $$P_{id,w}(q)=1+q.$$

\subsection{The $52431$-avoiding case}

When $w$ avoids $52431$ instead, we use the resolution
$$Z^\prime_w:=\{(G_\bullet,F_\bullet)\in
\mathcal{F}_{r^\prime_1,\ldots,r^\prime_k}(\mathbb{C}^n) \times X_w \mid
G_{r^\prime_i}\supseteq (F_{p_i}+ E_{q_i}) \ \forall i\},$$ where
$r^\prime_i:=p_i+q_i-r_i$.  Arguments similar to the above show that,
if we let $i$ be the index so that $(p_i,q_i)$ does not give an
inclusion element, the fiber $\pi_2^{-1}(e_{id})$ is
$$\{(G_\bullet,E_\bullet)\mid G_{r^\prime_j}=E_{r^\prime_j}\
\mathrm{for}\ j\neq i\ \mathrm{and}\ E_{r^\prime_i-1}\subseteq
G_{r^\prime_i}\subseteq E_{r^\prime_i+1}\}.$$

Hence the fiber over $e_{id}$ is isomorphic to $\mathbb{P}^1$ and
$P_{id,w}(q)=1+q$ by the same argument as above.

\section{Necessity in the $3412$ containing case}
\label{sect:not-covex}
In this section we treat the case where $w$ contains a $3412$ pattern.
Our strategy in this case is to use another resolution of
singularities given by Cortez~\cite{Cor2}.  We will again apply the
Decomposition Theorem \cite{BeilBernDel} to this resolution, but in
this case the calculation is more complicated as the fiber at $e_{id}$
will no longer always be isomorphic to $\mathbb{P}^1$.  When the fiber
at $e_{id}$ is not $\mathbb{P}^1$, we will need to identify the
image of the exceptional locus, which turns out to be irreducible, and calculate
the generic fiber over the image of the exceptional locus as well as the fiber over
$e_{id}$.  We then follow Polo's strategy in \cite{Polo} to calculate
that $P_{id,w}(q)=1+q^h$, where $h$ is the minimum height of a $3412$
embedding as defined below.

\subsection{Cortez's resolution}

We begin with some definitions necessary for defining a variety $Z$
and a map $\pi_2: Z\rightarrow X_w$ which we will show is our
resolution of singularities.  Our notation and terminology generally
follows that of Cortez~\cite{Cor2}.  Given an embedding
$i_1<i_2<i_3<i_4$ of $3412$ into $w$, we call $w(i_1)-w(i_4)$ its {\bf
height ({\it hauteur})}, and $w(i_2)-w(i_3)$ its {\bf amplitude}.  Among all
embeddings of $3412$ in $w$, we take the ones with minimum height, and
among embeddings of minimum height, we choose one with minimum
amplitude.  As we will be continually referring this particular
embedding, we denote the indices of this embedding by $a<b<c<d$ and
entries of $w$ at these indices by $\alpha=w(a)$, $\beta=w(b)$,
$\gamma=w(c)$, and $\delta=w(d)$.  We let $h=\alpha-\delta$ be the
height of this embedding.

Let $\alpha^\prime$ be the largest number such that
$w^{-1}(\alpha^\prime)<w^{-1}(\alpha^\prime-1)<\cdots<w^{-1}(\alpha+1)<w^{-1}(\alpha)$
and $\delta^\prime$ the smallest number such that
$w^{-1}(\delta)<w^{-1}(\delta-1)<\cdots<w^{-1}(\delta^\prime)$.  Also let $a^\prime=w^{-1}(\alpha^\prime)$ and $d^\prime=w^{-1}(\delta^\prime)$.  Now
let $\kappa=\delta^\prime+\alpha^\prime-\alpha$, let $I$ denote the set of
simple transpositions
$\{s_{\delta^\prime},\cdots,s_{\alpha^\prime-1}\}$, and let $J$ be
$I\setminus\{s_\kappa\}$.  Furthermore, let $v=w_0^Jw_0^Iw$, where $w_0^J$
and $w_0^I$ denote the longest permutations in the parabolic subgroups
of $S_n$ generated by $J$ and $I$ respectively.

As an example, let $w=817396254\in S_9$; its graph is in
figure~\ref{fig:perm817396254}.  Then $a=3$, $b=5$, $c=7$, and $d=8$,
while $\alpha=7$, $\beta=9$, $\gamma=2$, and $\delta=5$.  We also have
$h=2$, $\alpha^\prime=8$ and $\delta^\prime=4$.  Hence $\kappa=5$ and
$v=514398276$.

\begin{figure}[htbp]
\begin{center}

\input{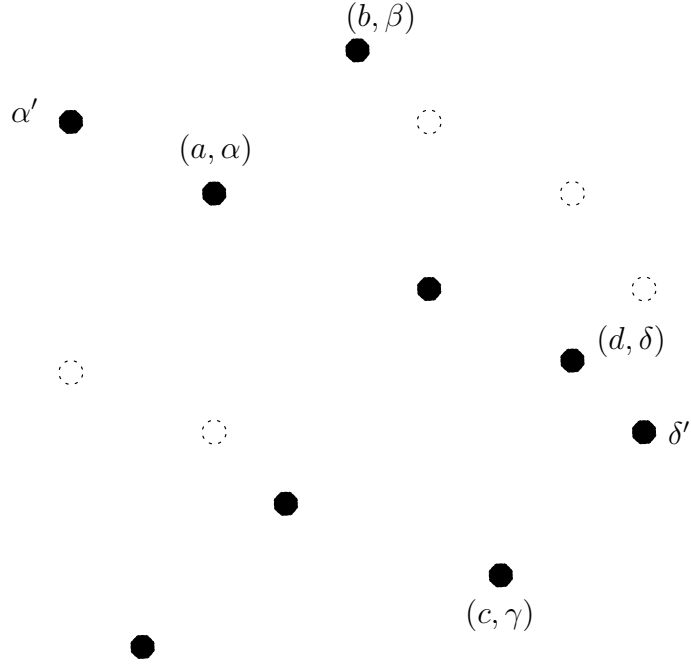}
\vspace{0.3in}
\caption{The graph of $w=817396254$ in black, labelled.  The points of
the graph of $v=514398276$ which are different from $w$ are in clear
circles.}
\label{fig:perm817396254}
\end{center}
\end{figure}

Now consider the variety $Z=P_I \times^{P_J} X_v$.  By definition, $Z$
is a quotient of $P_I \times X_v$ under the free action of $P_J$ where
$q\cdot(p,x)=(pq^{-1},q\cdot x)$ for any $q\in P_J$, $p\in P_I$, and
$x\in X_v$.  In the spirit of Magyar's realization \cite{Mag} of full
Bott-Samelson varieties as configuration varieties, we can also
consider $Z$ as the configuration variety

$$\{(G,F_\bullet)\in Gr_\kappa(\mathbb{C}^n) \times X_w \mid
E_{\delta^\prime-1}\subseteq G\subseteq E_{\alpha^\prime}
\ \mathrm{ and }\ \dim(F_i \cap G)\geq r_v(i,\kappa)\}.
\footnote{The statement of this geometric description
in~\cite{Cor2} has a typographical error.}$$

By the definition of $v$, $r_v(i,\kappa)=r_w(i,\alpha^\prime)$ for
$i<w^{-1}(\alpha-1)$, $r_v(i,\kappa)=r_w(i,\alpha^\prime)-j$ when
$w^{-1}(\alpha-j)\leq i<w^{-1}(\alpha-j-1)$, and
$r_v(i,\kappa)=r_w(i,\alpha^\prime)-\alpha^\prime+\kappa$ when $i\geq d^\prime$.
The last condition is automatically satisfied since, as $G\subseteq
E_{\alpha^\prime}$, we always have $\dim(G\cap F_i)\geq
\dim(E_{\alpha^\prime}\cap F_i)-(\alpha^\prime-\kappa)\geq
r_w(i,\alpha^\prime)-\alpha^\prime+\kappa$.

Cortez \cite{Cor2} introduced the variety $Z$ along with several other
varieties (constructed by defining
$\kappa=\delta^\prime+\alpha^\prime-\alpha+i-1$ for $i=1,\ldots,h$) to
help in describing the singular locus of Schubert
varieties\footnote{Cortez's choice of $3412$ embedding in \cite{Cor2}
is slightly different from ours.  For technical reasons she chooses
one of minimum amplitude among those satisfying a condition she calls
``well-filled'' ({\it bien remplie}).  As she notes, $3412$ embeddings
of minimum height are automatically ``well-filled''.}.  A virtually
identical proof would follow from analyzing the resolution given by
$i=h$ instead of $i=1$ as we are doing, but the other choices of $i$
give maps which are harder to analyze as they have more complicated
fibers.

The variety $Z$ has maps $\pi_1: Z\rightarrow P_I/P_J\cong
Gr_{\alpha^\prime-\alpha+1}(\mathbb{C}^{\alpha^\prime-\delta^\prime+1})$
sending the orbit of $(p,x)$ to the class of $p$ under the right
action of $P_J$ and $\pi_2: Z\rightarrow X_w$ sending the orbit of
$(p,x)$ to $p\cdot x$.  Under the configuration space description,
$\pi_1$ sends $(G,F_\bullet)$ to the point in
$Gr_{\alpha^\prime-\alpha+1}(\mathbb{C}^{\alpha^\prime-\delta^\prime+1})$
corresponding to the plane $G/E_{\delta^\prime-1}\subseteq
E_{\alpha^\prime}/E_{\delta^\prime-1}$, and $\pi_2$ sends
$(G,F_\bullet)$ to $F_\bullet$.  The map $\pi_1$ is a fiber bundle
with fiber $X_v$, and, by~\cite[Prop. 4.4]{Cor2}, the map $\pi_2$ is
surjective and birational. (In our case where the singular locus of
$X_w$ has only one component, the latter statement is also a
consequence our proof of Lemma~\ref{lem:excgeom}.)

In general $Z$ is not smooth; hence $\pi_2$ is only a partial
resolution of singularities.  However, we show in
Section~\ref{sect:lemmas} the following.

\begin{lemma}
\label{lem:ressm}
Suppose $w$ avoids $463152$ and the singular locus of $X_w$ has only
one irreducible component.  Then $Z$ is smooth.
\end{lemma}

\subsection{Fibers of the resolution}
\label{sect:nonvexfibers}

We now describe of the fibers of $\pi_2$.  To highlight the main flow
of the argument, proofs of individual lemmas will be deferred to
Section~\ref{sect:lemmas}.  Define $M=\max\{p\mid p<c,
w(p)<\delta^\prime\}\cup \{a\}$ and $N=\max\{p\mid
w(p)<\delta^\prime\}$.

\begin{lemma}
\label{lem:fibgeom}
The fiber of $\pi_2$ over a flag $F_\bullet$ is $$\{G\in
Gr_\kappa(\mathbb{C}^n) \mid E_{\delta^\prime-1} + F_M \subseteq G
\subseteq E_{\alpha^\prime} \cap F_N\}.$$
\end{lemma}

Now we focus on the fiber at the identity, and show that it is
isomorphic to $\mathbb{P}^h$.  Since the flag corresponding to the
identity is $E_\bullet$, it suffices by the previous lemma to show
that $\dim(E_{\delta^\prime-1} + E_M)=\kappa-1$ and
$\dim(E_{\alpha^\prime}\cap E_N)=\kappa+h$.

\begin{lemma}
\label{lem:fiblb}
Suppose that the singular locus of $X_w$ has only one component and
$w$ avoids $546213$.  Then $\dim(E_{\delta^\prime-1} + E_M)=\kappa-1$.
\end{lemma}

\begin{lemma}
\label{lem:fibub}
Suppose that the singular locus of $X_w$ has only one component and
$w$ avoids $465132$.  Then $\dim(E_{\alpha^\prime} \cap E_N)=\kappa+h$.
\end{lemma}

In the case where $h=1$, these are all the geometric facts we need.
When $h>1$, we identify the image of the exceptional locus as $X_u$
for a particular permutation $u$ of length $\ell(u)=\ell(w)-h$.  We
then show that the fiber over a generic point of $X_u$ is isomorphic
to $\mathbb{P}^{h-1}$.

First we describe the image of the exceptional locus geometrically.

\begin{lemma}
\label{lem:excgeom}
Suppose the singular locus of $X_w$ has only one component, and $h>1$.
Then the image of the exceptional locus of $\pi_2$ is
$\{F_\bullet\mid\dim(E_{\delta^\prime-1}\cap
F_M)>r_w(M,\delta^\prime-1)\}$.
\end{lemma}

Now let $\sigma\in S_n$ be the cycle $(\gamma, \delta+1, \delta+2,
\ldots, \alpha=\delta+h)$, and let $u=\sigma w$.  We show the
following.

\begin{lemma}
\label{lem:exccomb}
Assume that the singular locus of $X_w$ has only one component, that
$h>1$, and that $w$ avoids $526413$.  Then the image of the
exceptional locus of $\pi_2$ is $X_u$, $\ell(w)-\ell(u)=h$, and the
generic fiber over $X_u$ is isomorphic to $\mathbb{P}^{h-1}$.
\end{lemma}

\subsection{Calculation of $P_{id,w}(q)$}

We now have all the geometric information we need to calculate
$P_{id,w}(q)$, following the methods of Polo
\cite{Polo}.  The Decomposition Theorem \cite{BeilBernDel} shows that
$$H_{z,\pi_2}(q)=P_{z,w}(q)+\sum_{z\leq v<w}
q^{\ell(w)-\ell(v)}E_v(q)P_{z,v}(q),$$ where
$$H_{z,\pi_2}(q)=\sum_{i\geq0} q^i\dim H^{2i}(\pi_2^{-1}(e_z)),$$ and
$E_v(q)$ are some Laurent polynomials in $q^{\frac{1}{2}}$, depending
on $v$ and $\pi_2$ but not $z$, which have positive integer
coefficients and satisfy the identity $E_v(q)=E_v(q^{-1})$.
\footnote{For those readers familiar with the Decomposition Theorem:
No local systems appear in the formula since $X_w$ has a
stratification, compatible with $\pi_2$, into Schubert cells, all of
which are simply connected.}

When $h=1$, the fiber of $\pi_2$ at $e_{id}$ is isomorphic to
$\mathbb{P}^1$, and so by same argument as in Section~\ref{sect:53241-avoid},
$P_{id,w}(q)=1+q$.

For $h>1$, let $u$ be the permutation specified above.  For any $x$
with $x\leq w$ and $x\not\leq u$, $\pi_2^{-1}(e_x)$ is a point, so
$X_w$ is smooth at $e_x$, and $H_{x,w}(q)=1=P_{x,w}(q)$.  Therefore,
by induction downwards from $w$, $E_x(q)=0$ for any $x$ with $x\leq w$
and $x\not\leq u$.

Now we calculate $E_u(q)$.  From the above it follows that
$H_{u,\pi_2}(q)=P_{u,w}(q)+q^\frac{h}{2}E_u(q)$.  Since
$H_{u,\pi_2}(q)-P_{u,w}(q)$ has nonnegative coefficients and $\deg
P_{u,w}(q)\leq (h-1)/2 < h-1$, $$P_{u,w}(q)=1+\cdots+q^{s-1}$$ for some $s$,
$1\leq s\leq h-1$.  Then $q^\frac{h}{2}E_u(q)=q^s+\cdots+q^{h-1}$, so
$E_u(q)=q^{s-\frac{h}{2}}+\cdots+q^{\frac{h}{2}-1}$.  Since
$E_u(q^{-1})=E_u(q)$, $s=1$, so
$$q^\frac{h}{2}E_u(q)=q+\cdots+q^{h-1}.$$

To calculate $P_{id,w}(q)$, note that $H_{id,\pi_2}=1+q+\cdots+q^h$,
so
\begin{eqnarray*}
  P_{id,w}(q) & = & 
  H_{id,\pi_2}(q)-\sum_{x\leq w} q^{\frac{\ell(w)-\ell(x)}{2}}E_x(q)P_{id,x}(q)\\
& = & 1+\cdots+q^h-(q+\cdots+q^{h-1})P_{id,u}(q)+\sum_{x<u}
  q^{\frac{\ell(w)-\ell(x)}{2}}E_x(q)P_{id,x}(q).
\end{eqnarray*}
Evaluating at $q=1$, we see that
$$P_{id,w}(1)=h+1-(h-1)P_{id,u}(1)-\sum_{x<u} E_x(1)P_{id,x}(1).$$  Since
$P_{id,w}(1)\geq 2$, $P_{id,x}(1)$ is a positive integer for all $x$, and
$E_x(1)$ is a nonnegative integer for all $x$, we must have that
$P_{id,u}(1)=1$ and $E_x(1)=0$ for all $x<u$.  Therefore,
$P_{id,u}(q)=1$ and $E_x(q)=0$ for all $x<u$, and $$P_{id,w}(q)=1+q^h.$$

Readers may note that the last computation is essentially identical to
the one given by Polo in the proof of \cite[Prop. 2.4(b)]{Polo}.  In
fact, in this case the resolution we use, due to Cortez \cite{Cor2},
is very similar to the one described by Polo.

\begin{remark}
\label{rem:comb-calc}
We could have used a simultaneous Bialynicki-Birula cell
decomposition~\cite{BialI,BialII,Car} of the $Z$ and $X_w$, compatible
with the map $\pi_2$, to combinatorialize the above computation,
turning many geometrically stated lemmas into purely combinatorial
ones.  To be specific, for any $u$, the number $H_{u,\pi_2}(1)$ is the
number of factorizations $u=\sigma\tau$ such that $\tau\leq v$,
$\sigma\in W_I$, and $\sigma$ is maximal in its right $W_J$ coset.
(The last condition can be replaced by any condition that forces us to
pick at most one $\sigma$ from any $W_J$ coset.)  This observation
does not simplify the argument; the combinatorics required to
determine which factorizations of the identity satisfy these
conditions are exactly the same as the combinatorics used above to
calculate the fiber of $\pi_2$ at the identity.  It should also be
possible to combinatorially calculate $H_{u,\pi_2}(q)$ by attaching
the appropriate statistic to such a factorization.  If $Z$ were the
full Bott-Samelson resolution, the result would be Deodhar's approach
\cite{Deo} to calculating Kazhdan-Lusztig polynomials, and the
aforementioned statistic would be his defect statistic.  However, when
$Z$ is some other resolution, even one ``of Bott-Samelson type,'' no
reasonable combinatorial description of the statistic appears to be
known.
\end{remark}

\section{Lemmas}
\label{sect:lemmas}

In this section we give proofs for the lemmas of
Sections~\ref{sect:covex} and \ref{sect:not-covex}.  We begin with
Lemma~\ref{lem:covexonecomp}.

\begin{lemma:covexonecomp}
Suppose the singular locus of $X_w$ has only one component.  If $w$
contains both $53241$ and $52431$, then $w$ contains $632541$.
\end{lemma:covexonecomp}

\begin{proof}
Let $a<b<c<d<e$ be an embedding of $53241$, and
$a^\prime<b^\prime<c^\prime<d^\prime<e^\prime$ an embedding of
$52431$.  Since $b<d$ and $w(b)<w(d)$, there must be an element
$(p,q)$ of the coessential set such that $b<p<d$ and $w(b)<q<w(d)$.
This cannot be an inclusion element since $a<p$ but $w(a)>q$, and
$q<e$ but $w(e)>p$.  We also have $c<d$ and $w(c)<w(d)$, also inducing
an element of the coessential set which is not an inclusion element.
Since the singular locus of $X_w$ has only one component, this element
must also be $(p,q)$.  The pairs $b^\prime<c^\prime$ and
$b^\prime<d^\prime$ also each induce an element of the coessential set
which is not an inclusion element; hence these must also be the same
as $(p,q)$.  Therefore, $b<c<p<c^\prime<d^\prime$, and
$w(c)<w(b)<q<w(d^\prime)<w(c^\prime)$.

If $a^\prime>b$ and $w(a)<w(c^\prime)$, then there must be an element
$(p^\prime,q^\prime)$ of the coessential set with
$a<b<p^\prime<a^\prime<c^\prime$ and
$w(b)<w(a)<q^\prime<w(c^\prime)<w(a^\prime)$.  We now have
$p^\prime<a^\prime<p$ but $q<a\leq q^\prime$, contradicting $w$ being
covexillary.  Therefore, $a^\prime<b$ or $w(a)>w(c^\prime)$.  Similarly,
$e>d^\prime$ or $w(e^\prime)<w(c)$.  Let $a^{\prime\prime}$ be $a$ if
$w(a)>w(c^\prime)$ and $a^\prime$ if $a^\prime<b$, and
$e^{\prime\prime}$ be $e$ if $e>d^\prime$ and $e^\prime$ if
$w(e^\prime)<w(c)$.

Now $a^{\prime\prime}<b<c<c^\prime<d^\prime<e^{\prime\prime}$ is an
embedding of $632541$ in $w$.

\end{proof}

Recall that $(p,q)=(p_i,q_i)$ is the unique element of the coessential
set which is not an inclusion element, and $r=r_i=r_w(p,q)$.
Furthermore, $(p_{i-1},q_{i-1})$ is the immediately preceeding element
of the coessential set, and
$r_{i-1}=r_w(p_{i-1},q_{i-1}=\min(p_{i-1},q_{i-1})$.

\begin{lemma:covexfibub}
Suppose $w$ avoids $653421$ (and $3412$).  Then $\min\{p,q\}=r+1$.
\end{lemma:covexfibub}

\begin{proof}
Suppose that $r\leq\min\{p,q\}-2$; we show that in that case we have an
embedding of $3412$ or $653421$.  Since $r\leq p-2$, there exist
$a<b\leq p$ with $w(a),w(b)>q$.  Note that $w(a)>w(b)$, as, otherwise,
$a<b<p<w^{-1}(q+1)$ would be an embedding of $3412$.  Similarly, since
$r\leq q-2$, there exist $d>c>p$ with $w(d),w(c)\leq q$, and we have
$w(c)>w(d)$ since $w^{-1}(q)<p+1<c<d$ is an embedding of $3412$
otherwise.  Furthermore, if $b>w^{-1}(q)$, then $w(c)<w(p)$, as
otherwise $w^{-1}(q)<b<p<c$ would be an embedding of $3412$, and if
$w(b)<w(p+1)$, then $c>w^{-1}(q+1)$ to avoid $b<p+1<c<w^{-1}(q+1)$
being a similar embedding.

Now we have up to four potential cases depending on whether
$b<w^{-1}(q)$ or $b>w^{-1}(q)$, and whether $w(b)>w(p+1)$ or
$w(b)<w(p+1)$.  In each case we produce an embedding of $653421$.  If
$b<w^{-1}(q)$ and $w(b)>w(p+1)$, then $a<b<w^{-1}(q)<p+1<c<d$ is such
an embedding.  If $b<w^{-1}(q)$ and $w(b)<w(p+1)$, then we use
$a<b<w^{-1}(q)<q^{-1}(q+1)<c<d$.  If $b>w^{-1}(q)$ and $w(b)>w(p+1)$,
then we use $a<b<p<p+1<c<d$.  Finally, if $b>w^{-1}(q)$ and
$w(b)<w(p+1)$, $a<b<p<w^{-1}(q+1)<c<d$ produces the desired embedding.
\end{proof}

\begin{lemma:covexfiblb}
Suppose $w$ avoids $53241$ (and $3412$).  Then $r_{i-1}=r-1$.
\end{lemma:covexfiblb}

\begin{proof}

  We treat the two cases where $w(p)=q$ and $w(p)\neq q$ separately.
  First suppose $w(p)=q$.  Suppose for contradiction that
  $r_{i-1}<r-1$.  Then there must exist an index $b\neq p$ which
  contributes to $r=r_w(p,q)$ but not to
  $r_{i-1}=r_w(p_{i-1},q_{i-1})$.  This happens when $b\leq p$ and
  $w(b)\leq q$, but $b>p_{i-1}$ or $w(b)>q_{i-1}$.  Since $b<p$ and
  $w(b)<w(p)=q$, there must be an element $(p_j,q_j)$ of the
  coessential set such that $b\leq p_j<p$ and $w(b)\leq q_j<q$.  But
  then we have that $p_j>p_{i-1}$ or $q_j>q_{i-1}$, contradicting the
  definition of $(p_{i-1},q_{i-1})$ as the next element smaller than
  $(p_i,q_i)$ in our total ordering of the coessential set.
  Therefore, we must have $r_{i-1}=r_i-1$.

Now suppose $w(p)\neq q$.  Since $r<p$ and $r<q$, there exists $b<p$
with $w(b)>q$ and $c>p$ with $w(c)<q$.  Note that we cannot have both
$w(b)<w(p+1)$ and $c<w^{-1}(q+1)$, as, otherwise,
$b<p+1<c<w^{-1}(q+1)$ would be an embedding of $3412$.  It then
follows that we cannot have both $b<w^{-1}(q)$ and $w(c)<w(p)$; when
$w(b)>w(p+1)$, $b<w^{-1}(q)$ and $w(c)<w(p)$ imply that
$b<w^{-1}(q)<p<p+1<c$ is an embedding of $53241$, and when
$c>w^{-1}(q+1)$, $b<w^{-1}(q)$ and $w(c)<w(p)$ imply that
$b<w^{-1}(q)<p<w^{-1}(q+1)<c$ is an embedding of $53241$.  Therefore,
$b>w^{-1}(q)$ or $w(c)>w(p)$, and we now treat these two cases
separately.

Suppose $b>w^{-1}(q)$.  We must have $w(c)<w(p)$ in this case, because
otherwise $w^{-1}(q)<b<p<c$ would be an embedding of $3412$.  Let
$a=\min\{b\mid w^{-1}(q)<b<p, w(b)>q\}$.  We show that, for all
$b^\prime$ with $a\leq b^\prime<p$, $w(b^\prime)>q$.  First, we cannot
have both $w(a)<w(p+1)$ and $c<w^{-1}(q+1)$, as $a<p+1<c<w^{-1}(q+1)$
would be an embedding of $3412$ otherwise.  Now, if
$w(b^\prime)<w(p)$, then $w^{-1}(q)<a<b^\prime<p$ is an embedding of
$3412$, and if $w(p)<w(b^\prime)<q$, then either $a<b^\prime<p<p+1<c$
or $a<b^\prime<p<w^{-1}(q+1)<c$ would be an embedding of $53241$,
depending on whether $w(a)>w(p+1)$ or $c>w^{-1}(q+1)$.

We have now established that there is an element of the coessential
set at $(a-1,q)$.  Since this shares its second coordinate with
$(p,q)$, and $w(b)>q$ for all $b$, $a<b<p$, there are no
elements of the coessential set in between, and
$(p_{i-1},q_{i-1})=(a-1,q)$, so that $r_{i-1}=r_w(a-1,q)$.  Now,
$r_w(a-1,q)=r_w(p,q)-\#\{j\mid a-1<j\leq p, w(j)\leq q\}$.
The latter list has just one element, namely $j=p$, so
$r_{i-1}=r_i-1$.

Now suppose $w(c)>w(p)$ instead.  Then we let $s=\min\{t\mid
w(p)<t<q, w^{-1}(s)>p\}$.  By arguments symmetric with the
above, for all $s^\prime$ with $s\leq s^\prime<q$,
$s^\prime>w(p)$.  Therefore, there is an element of the coessential
set at $(p,s-1)$, and this is the element immediately before
$(p,q)$ in the total ordering.  Furthermore,
$r_w(p,s-1)=r_w(p,q)-\#\{j\mid s-1<j<q, w^{-1}(j)\leq p\}$,
and the latter list has one element, namely $j=q$, so $r_{i-1}=r_i-1$.
\end{proof}

Before moving on to prove the lemmas of Section~\ref{sect:not-covex},
we prove the following two lemmas which will be repeatedly used
further.  As in Section~\ref{sect:not-covex}, $a<b<c<d$ is
an embedding of $3412$ of minimal amplitude among such embeddings of
minimal height, and $\alpha$, $\beta$, $\gamma$, and $\delta$
respectively denote $w(a)$, $w(b)$, $w(c)$, and $w(d)$.

For Lemmas~\ref{lem:ressm}, \ref{lem:onecompempty}, and
\ref{lem:MNempty}, we use the description of the singular locus given
in Section~\ref{sect:singlocusdesc}.  It is worth noting that, since
we only need to detect when the singular locus has more than one
irreducible component, it is also possible to prove these lemmas using
Lemma~\ref{lem:num-sing-comps} (which was originally
\cite[Sect. 13]{BilWar}).  Another alternate approach is first to
directly prove Theorem \ref{thm:pats} by using the condition of
avoiding all its patterns instead of the condition of having one
component in the singular locus in the lemmas and then to prove
Theorem \ref{thm:main} as a corollary.  Neither approach appears to
substantially reduce the need for detailed case-by-case analysis in
the proof of these lemmas.

\begin{lemma}
\label{lem:onecompempty}
Suppose the singular locus of $X_w$ has only one component.  Then the
following sets are empty.

\begin{enumerate}
\item[(i)] $\{p\mid p<a, w(p)>\beta\}$
\item[(ii)] $\{p\mid p<a, \alpha^\prime<w(p)<\beta\}$
\item[(iii)] $\{p\mid a<p<b, \alpha<w(p)<\beta\}$
\item[(iv)] $\{p\mid b<p<c, \alpha<w(p)<\beta\}$
\item[(v)] $\{p\mid b<p<c, \beta<w(p)\}$
\item[(vi)] $\{p\mid p<b, \delta^\prime<w(p)<\alpha\}$
\item[(vii)] $\{p\mid p>d, w(p)<\gamma\}$
\item[(viii)] $\{p\mid p>d, \gamma<w(p)<\delta^\prime\}$
\item[(ix)] $\{p\mid c<p<d, \gamma<w(p)<\delta\}$
\item[(x)] $\{p\mid b<p<c, \gamma<w(p)<\delta\}$
\item[(xi)] $\{p\mid b<p<c, w(p)<\gamma\}$
\item[(xii)] $\{p\mid p>c, \delta<w(p)<\alpha^\prime\}$
\end{enumerate}
\end{lemma}

Most of this lemma and its proof is implicitly stated by Cortez,
scattered as parts of the proofs of various lemmas
in~\cite[Sect. 5]{Cor2}.  The empty regions are illustrated in
Figure~\ref{fig:lemonecompemptyregions}.

\begin{figure}[htbp]
\begin{center}

\input{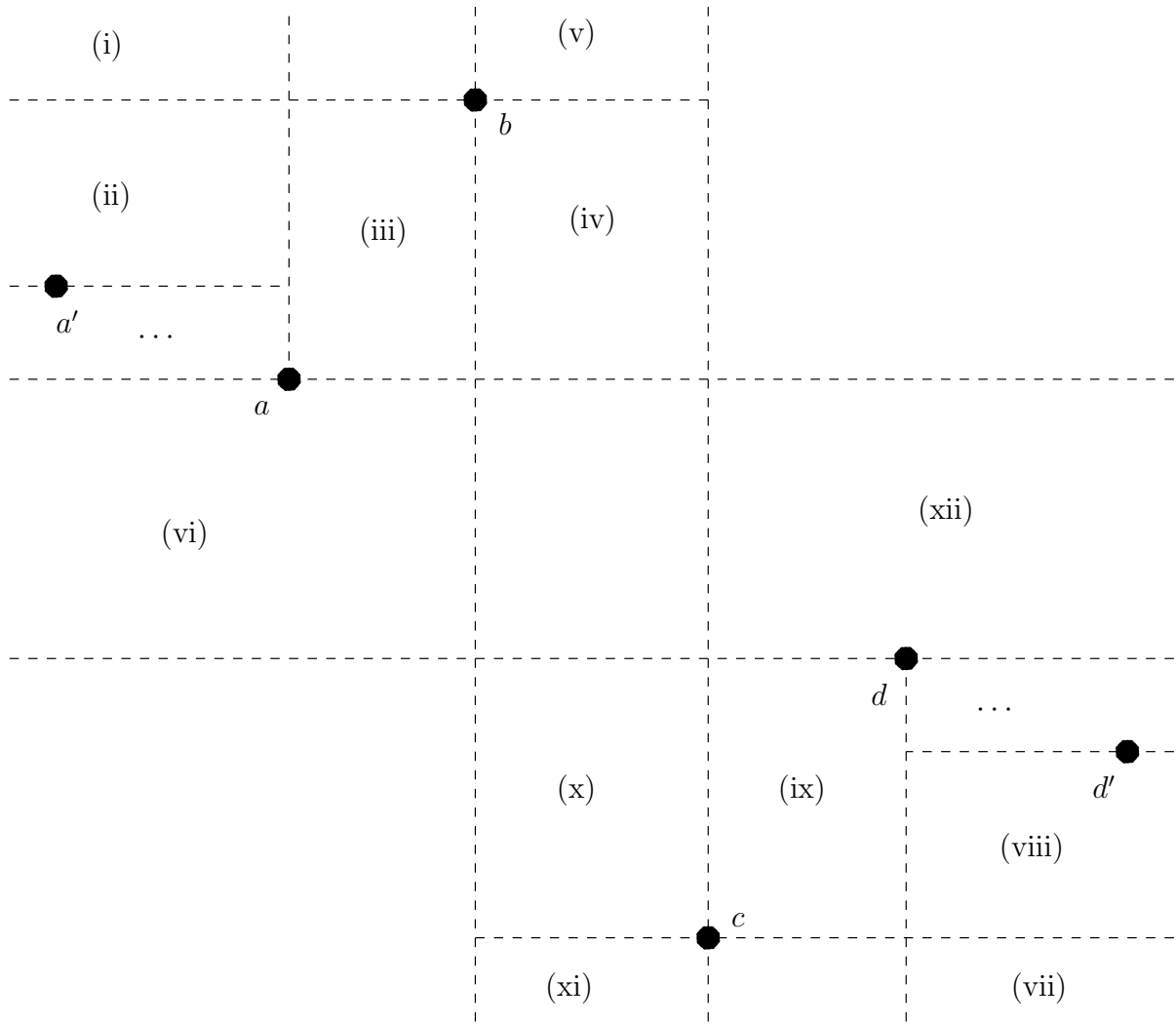}
\caption{The regions forced to be empty by Lemma~\ref{lem:onecompempty}.}
\label{fig:lemonecompemptyregions}
\end{center}
\end{figure}

\begin{proof}

If $p$ is in the set (vi), then $p<b<c<d$ is a $3412$ embedding with
height less than that of $a<b<c<d$.  If $p$ is in (iii) or (iv), then
$a<p<c<d$ is a $3412$ embedding of the same height but smaller
amplitude than $a<b<c<d$.  Similar arguments apply to (ix), (x), and
(xii).

Now we show that, if one of the other sets is nonempty, the singular
locus of $X_w$ must have at least two components.  Note that by the
emptiness of (iv), (vi), (x), and (xii) $a<b<c<d$ is a critical $3412$
embedding, and by the minimality of its height it must be reduced.

Suppose the set (v) is nonempty; let $p$ be the largest element of
(v).  Let $C=\{i\mid b<i<p, \delta<w(i)<\alpha\}$; if $C$ is nonempty,
then $i<p<c<d$ is a $3412$ embedding of smaller height than $a<b<c<d$
for any $i\in C$.  Now suppose $C$ is empty.  If the $A_2$ region
associated to $a<b<c<d$ is also empty, then $b<p<c<d$ is a reduced
critical $3412$ embedding.  The top critical region is empty by our
choice of $p$, the left critical region is empty by (iv) and the
emptiness of $C$, the bottom critical region is empty by (x), and the
right critical region is empty by (xii) and the emptiness of $A_2$;
furthermore it is reduced since $a<b<c<d$ is reduced.  Since $a\neq b$
and $b\neq p$, the components of the singular locus associated to
these critical $3412$ embeddings must be different, even if they are
of type I.  If $A_1$ or $B$ is empty, then $a<p<c<d$ is a reduced
critical $3412$ embedding.  The critical regions are empty by the
choice of $p$, the emptiness of $C$, and the emptiness of (iv), (vi),
(x), and (xii).  Since $b\neq p$, the only way the two critical $3412$
embeddings gave rise to the same component is for the component to be
a type I component using elements of both $A_1$ and $B$, but one of
these sets is empty in this case.  If $A_1$, $A_2$, and $B$ are all
nonempty, then the singular locus of $X_w$ must already have more than
one component.

Suppose (ii) is nonempty; let $e$ be the element of (ii) with the
smallest value of $\epsilon=w(e)$.  By the definition of
$\alpha^\prime$ and the emptiness of (iii) and (iv), either
$w^{-1}(\epsilon-1)>c$, or $\epsilon=\alpha^\prime+1$ and
$a^\prime<e<a$.

First we treat the case where $w^{-1}(\epsilon-1)>c$.  Let
$f=w^{-1}(\epsilon-1)$.  If $h>1$, then $e<b<w^{-1}(\alpha-1)<f$ is
a $3412$ embedding of height 1 and amplitude smaller than that of
$a<b<c<d$.  If $f>d$, then the same holds for $e<b<d<f$.  If $h=1$ and
$f<d$, then we have a type I component defined by $e$, $\{i\mid
e<i\leq a, \alpha\leq w(i)\leq\alpha^\prime\}$, which contains $a$, a
subset of $\{j\mid c<j<d, \alpha^\prime<w(j)<\epsilon\}$ that contains
$f$, and $d$.  This type I component cannot be the component of the
singular locus of $X_w$ associated to $a<b<c<d$, since $b\neq e$.

Now we treat the case where $\epsilon=\alpha^\prime+1$ and
$a^\prime<e<a$.  Let $i$ be the largest element of $\{i\mid a^\prime\leq
i<e, \alpha<w(i)\leq\alpha^\prime\}$.  Let $j$ be the smallest element
of $\{j\mid e<j\leq c, \gamma\leq w(j)<\delta\}$, a set which contains
$c$.  Let $k$ be the smallest element of $\{k\mid j<k\leq d,
w(j)<w(k)<w(i)\}$, a set which contains $d$.  Then $i<e<j<k$ is a
reduced critical $3412$ embedding.  The only portion of the critical
region not directly guaranteed empty by the definitions of $i$, $e$,
$j$, and $k$ is $\{m\mid e<m<j, \delta\leq w(m)<w(k)\}$; if $m$ is an
element of this set then $m<k<c<d$ is a $3412$ embedding of height
smaller than $a<b<c<d$.  Since $i\neq a$ and $e\neq b$, this must
produce a second component of the singular locus of $X_w$.  This shows
(ii) must be empty.

Suppose (i) is nonempty; let $e$ be the largest element of (i).  Then
the singular locus of $X_w$ has a type I component defined by $e$, a
set of which $a$ is the largest element, a set of which $b$ is the
largest element, and $w^{-1}(\alpha-1)$.

The proofs that (xi), (viii), and (vii) are empty are entirely
analogous to those for (v), (ii), and (i) respectively.

\end{proof}

For the following lemma, recall the definitions $M=\max\{p\mid p<c,
w(p)<\delta^\prime\}\cup \{a\}$ and $N=\max\{p\mid
w(p)<\delta^\prime\}$, given in Section~\ref{sect:not-covex}.

\begin{lemma}
\label{lem:MNempty}
Suppose the singular component of $X_w$ has only one component.  Then
\begin{enumerate}
\item[(i)] $a\leq M<b$.
\item[(ii)] $\{p\mid a<p<M, w(p)>\alpha^\prime\}$ is empty.
\item[(iii)] $c\leq N<d$.
\item[(iv)] $\{p\mid c<p<N, w(p)>\alpha^\prime\}$ is empty.
\end{enumerate}
\end{lemma}

This lemma is illustrated in Figure~\ref{fig:lemMNemptyregions}

\begin{figure}[htbp]
\begin{center}

\input{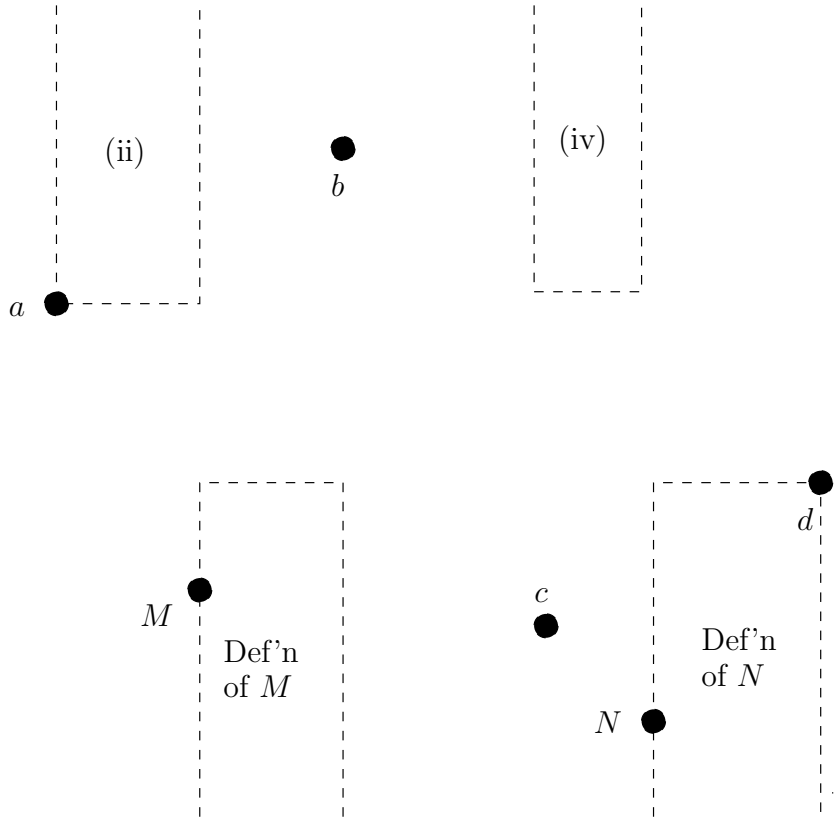}
\caption{The regions forced to be empty by Lemma~\ref{lem:MNempty}.}
\label{fig:lemMNemptyregions}
\end{center}
\end{figure}

\begin{proof}
We know that $a\leq M$ by definition, and $M<b$ by Lemma
\ref{lem:onecompempty} (x) and (xi).  Similarly, $c\leq N$ by
definition, and $N<d$ by Lemma \ref{lem:onecompempty} (vii) and (viii).

Now, assume for contradiction that $\{p\mid a<p<M,
w(p)>\alpha^\prime\}$ is nonempty.  Let $j=\max\{p\mid a<p<M,
\alpha<w(p)\}$.  By the definition of $j$ and Lemma
\ref{lem:onecompempty} (vi), $w(j+1)<\delta^\prime$.  Then
$a<j<j+1<w^{-1}(\alpha-1)$ is a reduced critical $3412$ embedding
defining a component of the singular locus in addition to the one
defined by $a<b<c<d$.

Similarly, suppose $\{p\mid c<p<N, w(p)>\alpha^\prime\}$ is nonempty.
Let $j=\max\{p\mid c<p<N, \alpha^\prime<w(p)\}$.  By the definition of
$j$ and Lemma \ref{lem:onecompempty} (xi), $w(j+1)<\delta^\prime$.
Then $w^{-1}(\delta+1)<j<j+1<d$ is a reduced critical $3412$ embedding
defining a component of the singular locus.
\end{proof}

We now proceed with the proof of the lemmas of
Section~\ref{sect:not-covex}, beginning with Lemma~\ref{lem:ressm}.

\begin{lemma:ressm}
Suppose the singular locus of $X_w$ has only one component and $w$
avoids $463152$.  Let $Z$ be constructed as above; then $Z$ is smooth.
\end{lemma:ressm}

\begin{proof}

Since $Z$ is a fibre bundle by the map $\pi_1$ over a smooth variety
(the Grassmannian) with fibre $X_v$, it is smooth if and only if $X_v$
is.

We show the contrapositive of our stated lemma by showing that, if
$X_v$ is not smooth and $w$ avoids $463152$, then the singular locus
of $X_w$ must have a component in addition to the one defined by the
reduced critical $3412$ embedding $a<b<c<d$.

Assume $X_v$ is singular.  We choose a component of its singular
locus.  This component has a combinatorial description as in
Section~\ref{sect:singlocusdesc}.

For convenience, we let $a_1=a^\prime=w^{-1}(\alpha^\prime)$,
$a_2=w^{-1}(\alpha^\prime-1)$, and so on with
$a_{\alpha^\prime-\alpha+1}=w^{-1}(\alpha)=a$.  Similarly, we let
$d_1=w^{-1}(\alpha-1)$, $d_2=w^{-1}(\alpha-2)$, and so on with $d_h=d$
and $d_{h+\delta-\delta^\prime}=d^\prime=w^{-1}(\delta^\prime)$.  We
also let $\mathcal{A}=\{a_1,\ldots,a_{\alpha^\prime-\alpha+1}\}$,
$\mathcal{D}_1=\{d_1,\ldots,d_{h-1}\}$,
$\mathcal{D}_2=\{d_h,\ldots,d_{h+\delta-\delta^\prime}\}$, and
$\mathcal{D}=\mathcal{D}_1\cup \mathcal{D}_2$.

First we handle the case where our chosen component of the singular
locus of $X_v$ is of type I.  If no index of the embedding into $v$
defining the component is in $\mathcal{A}$ or $\mathcal{D}$, then the
indices define an embedding of the same permutation into $w$, and the
sets required to be empty by the interval condition remain in exactly
the same positions.  The horizontal boundaries of these regions are
all above $\alpha^\prime$ or below $\delta^\prime$, so these regions
remain empty in $w$.  Therefore, the same embedding indices will
define a type I component of the singular locus of $X_w$.  This cannot
be the same as the component associated to the critical $3412$
embedding $a<b<c<d$; even if the component associated to $a<b<c<d$ is
of type I, it still must involve at least either $a$ or $d$, whereas
the component we just defined coming from the singular locus of $X_v$
involves neither.  Therefore, the singular locus of $X_w$ has at least
two components.

Now suppose our chosen type I component includes some index in
$\mathcal{A}$ or $\mathcal{D}$.  Let its defining embedding into $v$
be given by $i<j_1<\cdots<j_y<k_1<\cdots<k_z<m$.  Define the sets
$\mathcal{J}$ and $\mathcal{K}$ by $\mathcal{J}=\{j_1,\ldots,j_y\}$
and $\mathcal{K}=\{k_1,\ldots,k_z\}$.  We first show that one of
$\mathcal{A}$ and $\mathcal{D}$ contains no part of the embedding.  If
$a_r\in\mathcal{A}$ and $d_s\in\mathcal{D}$ are both in the embedding,
then since $a_r<d_s$ and $v(a_r)<v(d_s)$, $a_r\in\mathcal{J}$ and
$d_s\in\mathcal{K}$. Now we must have that $i<a_r$, and that
$v(i)>\alpha^\prime$, since, by definition, $v^{-1}(t)\in\mathcal{D}$
and hence $v^{-1}(t)>a_r$ whenever $d_s\leq t\leq\alpha^\prime$.
But then $i<a$ and $w(i)=v(i)>\alpha^\prime$, which is forbidden
by Lemma \ref{lem:onecompempty} (i) and (ii).

Therefore, we have two cases, one where $\mathcal{A}$ has some part of
our type I embedding but $\mathcal{D}$ does not, and one where
$\mathcal{D}$ has a part of our embedding but $\mathcal{A}$ does not.
We first tackle the case where $\mathcal{A}$ contains a part of the
embedding.  In this case, $i\in \mathcal{A}$, since otherwise $i<a$
and $w(i)>\alpha^\prime$, violating Lemma~\ref{lem:onecompempty} (i)
or (ii).  Having $i\in\mathcal{A}$ then implies that $m\not\in
\mathcal{A}$ and $\mathcal{J}\cap \mathcal{A}=\emptyset$ as follows.
First, we cannot have $m\in\mathcal{A}$ because, otherwise, any $r$
and $s$ satisfying $i<r<s<m$ would satisfy $v(r)>v(s)$, which
contradicts $i$ and $m$ being the first and last indices of a type I
embedding.  Second, $\mathcal{J}\cap\mathcal{A}$ must be empty
because, if $a_r\in\mathcal{A}$, $w(a_r)<w(k)<w(i)$ implies $i<k<a_r$ for
any $k$, contradicting $a_r\in\mathcal{J}$ for any type I embedding
starting with $i$.

We now have two subcases for the case where $\mathcal{A}$ has a part
of our type I embedding, depending on whether
$((\mathcal{K}\cup\{m\})\setminus\mathcal{A}$ contains an index less
than $b$.  If it does, then either $m<b$ or $k_s<b$ and
$w(k_s)<\delta^\prime$ for some $s$.  Either way, the forbidden region
for the type I embedding does not intersect $\{(p,q)\mid b<p,
\delta^\prime<q<\alpha^\prime\}$.  Therefore
$i<j_1<\ldots<j_y<k_1<\ldots<k_z<m$ defines a type I component of the
singular locus of $X_w$ as well as $X_v$.  The forbidden region may be
a little larger in $w$, but it does not acquire any points in the
graph of $w$.  This cannot be the same as any type I component of the
singular locus of $X_w$ associated to $a<b<c<d$ since both
$\mathcal{J}$ and $\mathcal{K}$ contain indices outside of the region
$B$ associated to $a<b<c<d$.

In the other case, since $m>b$, we must have $c\leq m<d$ by Lemma
\ref{lem:onecompempty} (x), (xi), (vii), and (viii).  One possibility
is that $c=m$.  In this case, taking the type I embedding in $v$ and
adding $\mathcal{D}_1$ to $\mathcal{K}$ gives a type I component of
the singular locus of $X_w$.  Both $\mathcal{J}$ and $\mathcal{K}$
contain indices outside of $B$, so this will also be a second
component of the singular locus of $X_w$.

If, on the other hand, $c\neq m$, then $c\in \mathcal{K}$ by the
following argument.  An example of this case is in
Figure~\ref{fig:typeIhardcase}.  By Lemma \ref{lem:onecompempty} (ix),
$v(m)=w(m)<\gamma$.  Furthermore, $j<c$ and $v(j)<\gamma$ for all
$j\in \mathcal{J}$ as follows.  If $j_r>c$ and $j_{r-1}<c$ (allowing
for $r=1$ in which case we define $j_0=i$), the forbidden region
$\{(p,q)\mid j_{r-1}<p<j_r, v(j_r)<q<v(i)\}$ for our type I embedding
contains $(c,\gamma)$ as $v(j_r)<\gamma$ by Lemma
\ref{lem:onecompempty} (ix).  If $v(j)\geq\gamma$ for some
$j\in\mathcal{J}$, then $v(k_z)>\gamma$, and hence $k_z<c$ by Lemma
\ref{lem:onecompempty} (ix); now the forbidden region $\{(p,q)\mid
k_z<p<m, v(m)<q<v(k_z)\}$ contains $(c,\gamma)$.

\begin{figure}[htbp]
\begin{center}

\input{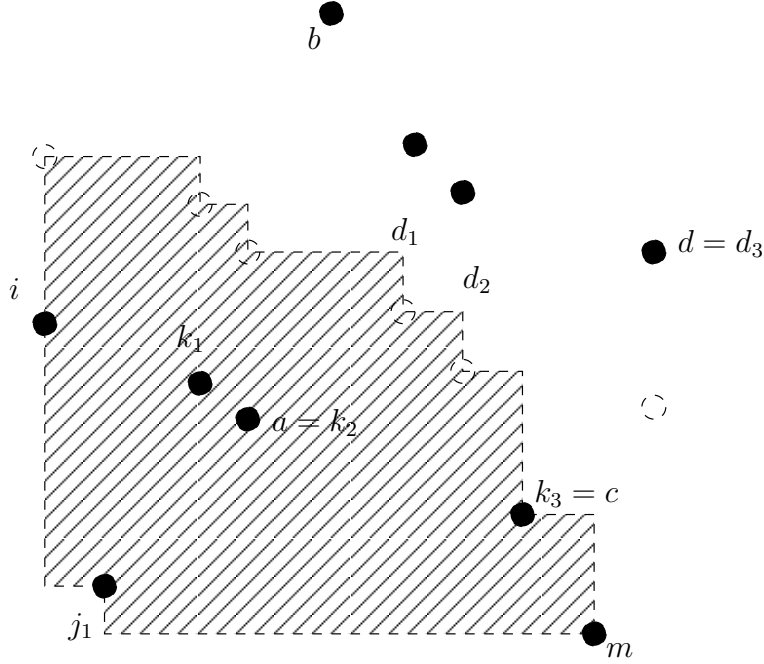}
\caption{The case of a type I configuration in $v$, using points in
$\mathcal{A}$, with $c<m<d$.  The hollow points are in $w$, and the
shaded region is the forbidden region of the associated configuration
in $w$.}
\label{fig:typeIhardcase}
\end{center}
\end{figure}

Recall that $(\mathcal{K}\setminus \mathcal{A})$ has no index less
than $b$ in the case under consideration.  Therefore, by Lemma
\ref{lem:onecompempty} (xi), no index $k\in \mathcal{K}$ satisfies
$k<c$, $v(k)<\gamma=v(c)$.  As $i<c<m$, $v(m)<\gamma<v(i)$, and $j<c$
and $v(j)<\gamma$ for all $j\in \mathcal{J}$, we must have $c\in
\mathcal{K}$ as otherwise $(c,\gamma)$ would be in a forbidden region.
Therefore, taking the type I embedding in $v$ and adding
$\mathcal{D}_1$ to $\mathcal{K}$ also gives a type I component of the
singular locus of $X_w$ distinct from any associated to $a<b<c<d$.

Now suppose $\mathcal{D}$ contains some part of the embedding but
$\mathcal{A}$ does not.  If $i\not\in \mathcal{D}$, then
$w(i)=v(i)>\alpha^\prime$, so by Lemma \ref{lem:onecompempty} (i) and
(ii), $i>a$.  If $i\in\mathcal{D}$, we also have $i>a$.  (Actually, we
cannot have $i\in \mathcal{D}$ but do not need this fact.)  Therefore,
$i<j_1<\cdots<j_y<k_1<\cdots<k_z<m$ also defines a type I component of
the singular locus of $X_w$, since, as the forbidden region does not
intersect $\{(p,q)\mid p\leq a, \delta^\prime<q<\alpha^\prime\}$,
no points of the graph of $w$ move into the forbidden region.  This
type I component can be the same as one associated to the critical
$3412$ embedding $a<b<c<d$, but only if $w$ has an embedding of
$463152$.

We have completed the case where our chosen component of the singular
locus of $X_v$ is of type I; now we move on to the case where it is of
type II.  Let $i<j<k<m$ be the reduced critical $3412$ embedding
associated to this component of the singular locus of $X_v$.  If none
of $i$, $j$, $k$, and $m$ are in $\mathcal{A}$ or $\mathcal{D}$, then
the critical regions are in the same place in both $v$ and $w$, and
they remain empty.  Therefore, they produce a component of the
singular locus of $X_w$ which must not be the same as the one
associated to $a<b<c<d$ as their reduced critical $3412$ embeddings
are different.

Now we first consider the case where $\mathcal{D}$ has a part of the
critical embedding but $\mathcal{A}$ does not.  If $i\in\mathcal{D}$,
then $i>a$, so the critical region as well as the regions $A$ and $B$
associated to $i<j<k<m$ do not intersect $\{(p,q)\mid p\leq a,
\delta^\prime\leq q\leq \alpha^\prime\}$, and $i<j<k<m$ is also a
reduced critical $3412$-embedding producing a type II component of the
singular locus of $X_w$.  If $j\in\mathcal{D}$, then $v(i)<k$, and,
since $i\not\in\mathcal{A}$ by assumption, $v(i)<\delta^\prime$.
Since $j>a$, we therefore also have that $\{(p,q)\mid p\leq
a,\delta^\prime\leq q\leq \alpha^\prime\}$ fails to intersect the
critical regions or the regions $A$ and $B$, and $i<j<k<m$ is a
critical $3412$ embedding producing a type II component of the
singular locus of $X_w$.  Otherwise, $i<d_1$ and $v(i)>\alpha^\prime$,
so by Lemma \ref{lem:onecompempty} (i) and (ii), $i>a$, implying that
$i<j<k<m$ produces a type II component of the singular locus of $X_w$.
Since $i<j<k<m$ is not $a<b<c<d$, we must have produced a second
component of the singular locus of $X_w$ in all of these cases.

Now suppose $\mathcal{A}$ has part of the critical embedding.  We
cannot have $k\in \mathcal{A}$ or $m\in \mathcal{A}$, since otherwise
we would have $i<j<a$ with $v(a)<v(i)<v(j)$, which forces
$j\not\in\mathcal{A}$.  Then $j<a$ and $w(j)=v(j)>\alpha^\prime$,
violating Lemma \ref{lem:onecompempty} (i) or (ii).  Therefore,
$j\in\mathcal{A}$ or $i\in\mathcal{A}$.

If $j\in \mathcal{A}$, then since $i<j$ and $v(i)<v(j)$,
$v(i)<\delta^\prime$, and so $v(k)<v(m)<\delta^\prime$.  Now if $k<b$,
$i<j<k<m$ is a reduced critical $3412$ embedding in $w$.  It may have
an element in its $A_2$ region in $w$ that when there is none in its
$A_2$ region in $v$, but in that case either the $B$ region is empty
or $w$ fails to avoid $463152$.  When $w$ avoids $463152$, $i<j<k<m$
produces a second type II component of the singular locus of $X_w$.
If $j\in \mathcal{A}$ and $k>b$, then by Lemma \ref{lem:onecompempty}
(x) and (xi), $k\geq c$.  Moreover, we cannot have $k=c$ as, in that
case, $c<m$ and $\gamma<v(m)=w(m)<\delta^\prime$, violating Lemma
\ref{lem:onecompempty} (viii) or (ix).  Therefore, $c<k<m$, and, as
$m<\delta^\prime$, $m<d$ by Lemma \ref{lem:onecompempty} (vii) and
(viii).  Since $j<c<k$, we now must have that $v(i)=w(i)>\gamma$ in
order for $i<j<k<m$ to be a critical $3412$ embedding in $v$.  If
$h=1$ and hence $\mathcal{D}_1$ is empty, then $i<j<k<m$ is a critical
$3412$ embedding in $w$ with $A$ or $B$ empty as they are in $v$.  If
$h>1$, let $i^\prime=\max\{p\mid p<b, \gamma<w(p)\leq w(i)\}$; this
set is nonempty because $i$ is an element.  Then
$i^\prime<d_{h-1}<k<m$ is a reduced critical embedding of $3412$ in
$w$, and the component of the singular locus of $X_w$ it produces,
whether it is type I or type II, must be different from the one
associated to $a<b<c<d$.  This last case is illustrated in
Figure~\ref{fig:typeIIhardcase}.

\begin{figure}[htbp]
\begin{center}

\input{typeIIhardcase.pstex_t}
\caption{The case of a type II configuration in $v$, using points in
$\mathcal{A}$, with $h>1$ and $k>b$.  The hollow points are in $w$,
and the shaded regions are the critical regions of the associated
$3412$ embedding in $w$.}
\label{fig:typeIIhardcase}
\end{center}
\end{figure}

Finally we tackle the case where $i\in\mathcal{A}$.  If $m<b$, then
$i<j<k<m$ is a reduced critical $3412$ embedding in $w$.  Otherwise,
$m\geq c$ by Lemma \ref{lem:onecompempty} (x) and (xi), and hence
$v(m)<\delta^\prime$.  If $k<c$, then $v(k)<\delta^\prime$, so $k\leq
M$ by definition.  We then have $j<M$ with $v(j)>\alpha^\prime$, which
is forbidden by Lemma \ref{lem:onecompempty} (i), (ii), and (iii) and
Lemma \ref{lem:MNempty} (ii).  We cannot have $k=c$ since in that case
$\gamma<w(m)=v(m)<\delta^\prime$ and $m>c$, violating Lemma
\ref{lem:onecompempty} (viii) or (ix).  If $k>c$ then we have
$c<k<m<d$.  In this case $a<b<k<m$ is a critical $3412$ embedding in
$w$.  In particular, $\{p\mid b<p<k, \alpha<w(p)<\beta\}$ is empty by
Lemma \ref{lem:onecompempty} (iv) and Lemma \ref{lem:MNempty} (iv).
Since $k\neq c$ and $m\neq d$, the associated component of the
singular locus of $X_w$ must be different from the component
associated to $a<b<c<d$.

We have now shown that, unless $463152$ embeds in $w$, no matter what
singularity $X_v$ may have, it must produce a second component of the
singular locus of $X_w$, either directly or through the use of Lemma
\ref{lem:onecompempty} or Lemma \ref{lem:MNempty}.  Therefore, if the
singular locus of $X_w$ has only one component and $w$ avoids
$463152$, $X_v$, and hence $Z$, is nonsingular.

\end{proof}

Now we continue on to proving the lemmas of Section~\ref{sect:nonvexfibers}.

\begin{lemma:fibgeom}
The fiber of $\pi_2$ over a flag $F_\bullet$ is $$\{G\in
Gr_\kappa(\mathbb{C}^n) \mid E_{\delta^\prime-1} + F_M \subseteq G
\subseteq E_{\alpha^\prime} \cap F_N\}.$$
\end{lemma:fibgeom}

\begin{proof}

By definition of $Z$, $E_{\delta^\prime-1} \subseteq G \subseteq
E_{\alpha^\prime}$.  We need to show that $F_M\subseteq G$, that
$G\subseteq F_N$, and that any such subspace $G$ is in
$\pi_2^{-1}(F_\bullet)$.

To show that $F_M\subseteq G$, we show that $r_v(M,\kappa)=M$.  This is
equivalent to showing that $\{p\mid p\leq M, v(p)>\kappa\}$ is
empty, which is in turn equivalent to showing that $\{p\mid
p\leq M, \delta^\prime-1<w(p)<\alpha\}$ and $\{p\mid p\leq
M, \alpha^\prime<w(p)\}$ are both empty.  The first follows from
Lemma \ref{lem:onecompempty} (vi) since $M<b$ by Lemma
\ref{lem:MNempty} (i).  The second follows from Lemma
\ref{lem:onecompempty} (i) and (ii) and Lemma \ref{lem:MNempty} (ii).

Now we show $G\subseteq F_N$.  This means showing that $r_v(N,\kappa)=\kappa$,
or that $\{p\mid p>N, v(p)\leq \kappa\}$ is empty.  This is
equivalent to showing that $\{p\mid p>N,
w(p)<\delta^\prime\}$ and $\{p\mid p>N, \alpha\leq
w(p)\leq \alpha^\prime\}$ are both empty.  The first is empty by
the definition of $N$, and the second is empty by the definition of
$\alpha^\prime$.

To show that any $G$ satisfying $E_{\delta^\prime-1} + F_M \subseteq G
\subseteq E_{\alpha^\prime} \cap F_N$ is in $\pi_2^{-1}(F_\bullet)$,
we need to show that $\dim(G \cap F_j) \geq r_v(j,\kappa)$ for any $j$ with
$M<j<N$.  It suffices to show that $r_v(j,\kappa)=r_v(M,\kappa)=M$ when $M<j<c$,
and that $r_v(j,\kappa)=r_v(j-1,\kappa)+1$ when $c\leq j\leq N$.  Equivalently,
this means that $v(j) > \kappa$ when $M<j<c$ and $v(j) \leq \kappa$ when $c\leq
j\leq N$.

Since $v(j)\leq \kappa$ if and only if $w(j)<\delta^\prime$ or $\alpha\leq
w(j)\leq\alpha^\prime$, the first condition is clear from the
definition of $M$.  We also have that $N<d$ by Lemma
\ref{lem:MNempty} (iii), so we need that
$\{p\mid c<p<N, w(p)>\alpha^\prime\}$ is empty, which follows
from Lemma \ref{lem:MNempty} (iv).  Therefore, $r_v(j,\kappa)=r_v(j-1,\kappa)+1$
when $c\leq j\leq N$, and any $G$ satisfying $E_{\delta^\prime-1} +
F_M \subseteq G \subseteq E_{\alpha^\prime} \cap F_N$ is in
$\pi_2^{-1}(F_\bullet)$.

\end{proof}
\pagebreak
\begin{lemma:fiblb}
Suppose that the singular locus of $X_w$ has only one component and
$w$ avoids $546213$.  Then $\dim(E_{\delta^\prime-1} + E_M)=\kappa-1$.
\end{lemma:fiblb}

\begin{proof}

Since $r_v(M,\kappa)=M$, $c>M$, and $v(c)<\kappa$, $M=r_v(M,\kappa)+1\leq
r_v(c,\kappa)\leq \kappa$, so $M\leq \kappa-1$.  If $\alpha=\alpha^\prime$, then
$\delta^\prime=\alpha^\prime-\alpha+\delta^\prime=\kappa$, so
$\delta^\prime-1=\kappa-1$.  Otherwise, we need to show that $M=\kappa-1$.
Since $M\geq a$, so that $\{p\mid p>M,
\alpha\leq p\leq\alpha^\prime\}$ is empty, this is equivalent to
showing that $\{p\mid p>M,w(p)<\delta^\prime\}$ has only one
element, namely $c$.

By the definition of $M$, $\{p\mid M<p<c,
w(p)<\delta^\prime\}$ is empty.  Furthermore, by Lemma
\ref{lem:onecompempty} (vii), (viii), and (ix),
$\{p\mid p>d, w(p)<\gamma\}$, $\{p\mid p<d,
\gamma<w(p)<\delta^\prime\}$, and $\{p\mid c<p<d,
\gamma<w(p)<\delta\}$ are empty.  This leaves $\{p\mid c<p<d,
w(p)<\gamma\}$, which is empty since $\alpha^\prime\neq\alpha$ and
$w$ avoids $546213$.
\end{proof}

\begin{lemma:fibub}
Suppose that the singular locus of $X_w$ has only one component and
$w$ avoids $465132$.  Then $\dim(E_{\alpha^\prime} \cap E_N)=\kappa+h$.
\end{lemma:fibub}

\begin{proof}

First, note that $N\geq \kappa+h$, since $N=\#\{p\mid p<N,
v(p)\leq \kappa\}+\#\{p\mid p<N, v(p)>\kappa\}$, and the first
summand is $r_v(N,\kappa)=\kappa$, while the second summand is at least $h$
since the $h$ elements $b,w^{-1}(\alpha-1),\ldots,w^{-1}(\delta+1)$
are in the set.  If $\delta^\prime=\delta$, then
$\alpha^\prime=\kappa+\alpha-\delta=\kappa+h$.  Otherwise, we need to show that
$N=\kappa+h$.  This means showing that $\{p\mid p<N, v(p)>\kappa\}$ has
exactly $h$ elements, or, equivalently, that $\{p\mid p<N,
w(p)>\alpha^\prime\}$ contains only $b$.

We know that $\{p\mid c<p<N, w(p)>\alpha^\prime\}$ is empty
by Lemma \ref{lem:MNempty} (iv), and, by Lemma \ref{lem:onecompempty}
(i), (ii), and (iii), $\{p\mid p<a, w(p)>\beta\}$,
$\{p\mid p<a, \alpha^\prime<w(p)<\beta\}$, and $\{p\mid
a<p<b, \alpha<w(p)<\beta\}$ are empty.  This leaves $\{p\mid
a<p<b, w(p)>\beta\}$, which is empty since
$\delta^\prime\neq\delta$ and $w$ avoids $465132$.

\end{proof}

\begin{lemma:excgeom}
Suppose the singular locus of $X_w$ has only one component, and $h>1$.
Then the image of the exceptional locus of $\pi_2$ is
$$\{F_\bullet\mid\dim(E_{\delta^\prime-1}\cap
F_M)>r_w(M,\delta^\prime-1)\}.$$
\end{lemma:excgeom}

\begin{proof}

First we show that $r_w(N,\alpha^\prime)=\kappa+h-1$.  By the definition of
$N$, $r_v(N,\kappa)=\kappa$, so the two sets $\{p\mid p<N,
w(p)<\delta^\prime\}$ and $\{p\mid p<N, \alpha\leq
w(p)\leq \alpha^\prime\}$ have $\kappa$ elements combined.  Since $c<N$
by definition and $N<d$ by Lemma \ref{lem:onecompempty} (vii) and
(viii), $\{p\mid p<N, \delta^\prime\leq w(p)<\alpha\}$ has
precisely the $h-1$ elements
$w^{-1}(\alpha-1),\ldots,w^{-1}(\delta+1)$.  Therefore,
$\dim(E_{\alpha^\prime}\cap F_N)=r_w(N,\alpha^\prime)=\kappa+h-1$ generically.

Now we calculate $\dim(E_{\delta^\prime-1}+F_M)$.  Note that
\begin{equation*}
\dim(E_{\delta^\prime-1}+F_M)=\delta^\prime-1+M -\dim(E_{\delta^\prime-1}\cap F_M).
\end{equation*}
Generically,
$$\dim(E_{\delta^\prime-1}\cap F_M)=r_w(M,\delta^\prime-1),$$
and 
\begin{eqnarray*}
r_w(M,\delta^\prime-1)
& = & r_v(M,\kappa)-\#\{p\mid p<M, \alpha\leq w(p)\leq \alpha^\prime\} \\
& = & M-(\alpha^\prime-\alpha+1).
\end{eqnarray*}
Therefore, generically,
\begin{eqnarray*}
\dim(E_{\delta^\prime-1}+F_M)
& = & \delta^\prime-1+M-M+\alpha^\prime-\alpha+1 \\
& = & \delta^\prime+\alpha^\prime-\alpha \\
& = & \kappa.
\end{eqnarray*}

Recall that, by Lemma~\ref{lem:fibgeom}, the fiber over a flag $F_\bullet$ is
$$\{G\in Gr_\kappa(\mathbb{C}^n) \mid E_{\delta^\prime-1} + F_M
\subseteq G \subseteq E_{\alpha^\prime} \cap F_N\}.$$ Therefore, since
$\dim G=\kappa$, the fiber over $F_\bullet$ consists of the single
point corresponding to the subspace $E_{\delta^\prime-1}+F_M$
generically; here, the generic situation occurs whenever
$\dim(E_{\delta^\prime-1}\cap F_M)=r_w(M,\delta^\prime-1)$.  When
$h=1$, we also have that $\dim(E_{\alpha^\prime}\cap F_N)=\kappa$ in
the generic situation, and we also need $\dim(E_{\alpha^\prime}\cap
F_N)>\kappa$ in order for the fiber over $F_\bullet$ to consist of
more than a point.  However, when $h>1$, $\pi^{-1}(F_\bullet)$ has
more than one point whenever $\dim(E_{\delta^\prime-1}\cap
F_M)>r_w(M,\delta^\prime-1)$, so the image of the exceptional locus is
$$\{F_\bullet\mid\dim(E_{\delta^\prime-1}\cap
F_M)>r_w(M,\delta^\prime-1)\},$$ as desired.
\end{proof}

Recall that $u$ is defined by $u=\sigma w$, where $\sigma\in S_n$ is
the cycle $(\gamma, \delta+1, \delta+2, \ldots, \alpha)$.

\begin{lemma:exccomb}
Assume that $h>1$ and $w$ avoids $526413$.  Then the image of the
exceptional locus of $\pi_2$ is $X_u$, $\ell(w)-\ell(u)=h$, and the
generic fiber over $X_u$ is isomorphic to $\mathbb{P}^{h-1}$.
\end{lemma:exccomb}

\begin{proof}

Suppose $F_\bullet$ is in the image of the exceptional locus, and let
$X^\circ_x$ be the Schubert cell containing $F_\bullet$.  Our strategy
is to show using rank matrices that $x\leq u$.  As part of this proof,
we show that a certain region of the graph of $w$ is empty, which will
imply that $\ell(w)-\ell(u)=h$.

First we compare the rank matrices $r_u$ and $r_w$.  Let $R_1$ denote
the region $\{(p,q)\mid w^{-1}(\delta+1)\leq p<c, \gamma\leq
q<\delta+1\}$ and $R_i=\{(p,q)\mid w^{-1}(\delta+i-1)\leq
p<w^{-1}(\delta+i), \gamma\leq q<\delta+i\}$ when $1<i\leq h$.  Since
$u=t_h\cdots t_1 w$ where $t_1=(\gamma,\delta+1)$ and
$t_i=(\delta+i-1,\delta+i)$ when $1<i\leq h$, we get that
$r_u(p,q)=r_w(p,q)+1$ if $(p,q)$ is in $R_i$ for some $i$, and
$r_u(p,q)=r_w(p,q)$ otherwise.  Let $R$ denote the union
$R=\bigcup_{i=1}^h R_i$.  The region $R$ is drawn in
Figure~\ref{fig:exccombregion}.

\begin{figure}[htbp]
\begin{center}

\input{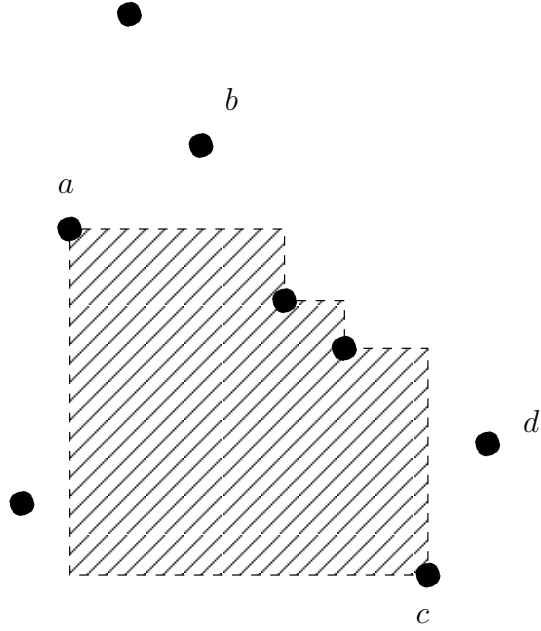}
\caption{The region $R$ ``between'' $u$ and $w$.}
\label{fig:exccombregion}
\end{center}
\end{figure}

Now we show that, when $(p,q)\in R$, then $r_w(p,q)$ is as small as
possible given that $r_w(M,\delta^\prime-1)=M$ and given that, as for
any permutation and any $p$ and $q$, $0\leq r_w(p+1,q)-r_w(p,q)\leq 1$
and $0\leq r_w(p,q+1)-r_w(p,q)\leq 1$.  To be precise, this means
that, assuming $(p,q)$ and $(p,q+1)$ are both in $R$,
$r_w(p,q)=r_w(p,q+1)-1$ if $\gamma\leq q<\delta^\prime$ and
$r_w(p,q)=r_w(p,q+1)$ otherwise, and, assuming $(p,q)$ and $(p+1,q)$
are both in $R$, $r_w(p,q)=r_w(p+1,q)-1$ if $a\leq p<M$ with
$r_w(p,q)=r_w(p+1,q)$ otherwise.

To prove the above claim, we need to show that $R$
contains no point in the graph of $w$, that $w(p)<\gamma$ when $a<p\leq
M$, and that $w^{-1}(q)<a$ when $\gamma<q\leq\delta^\prime-1$.  Since
$h>1$ and $w$ avoids $526413$, $\{p\mid a<p<b,
\gamma<w(p)<\delta^\prime\}$.  Also, $\{p\mid a<p<b, \delta^\prime\leq
w(p)<\alpha\}$ and $\{p\mid b\leq p<c, \gamma<w(p)<\delta^\prime\}$
are empty by Lemma \ref{lem:onecompempty} (vi) and (x).  The remaining
portion of $R$ contains no point in the graph of $w$ since, as
$a<b<c<d$ is a $3412$ embedding of minimal height,
$b<w^{-1}(\alpha-1)<\cdots<w^{-1}(\delta+1)<c$ by $a<b<c<d$.
Furthermore, using that $R$ contains no point of the
graph of $w$, $w(p)<\gamma$ when $a<p\leq M$ by Lemma
\ref{lem:MNempty} (ii), and $w^{-1}(q)<a$ when
$\gamma<q\leq\delta^\prime-1$ by Lemma \ref{lem:onecompempty} (viii)
and (ix).

Suppose $X^\circ_x$ is in the image of the exceptional locus, so that
$x\leq w$ and $r_x(M,\delta^\prime-1)\geq r_w(M,\delta^\prime-1)+1$.
We show that $r_x(p,q)\geq r_u(p,q)$ for all $p$ and $q$.  If $(p,q)$
is not in $R$, this follows since $x\leq w$.  For $(p,q)\in R$,
$r_w(p,q)$ is the minimum possible given that
$r_w(M,\delta^\prime-1)=M$.  Since
$r_x(M,\delta^\prime-1)>r_w(M,\delta^\prime-1)$, it follows that
$r_x(p,q)>r_w(p,q)$ for $(p,q)\in R$.  Therefore, $r_x(p,q)\geq
r_u(p,q)$ when $(p,q)$ is in $R$, and $x\leq u$.

Since the regions $R_i$ are empty, multiplication by each
transposition $t_1$,\ldots, $t_h$ decreases the length of $w$ by $1$,
and $\ell(u)=\ell(w)-h$.

Since $u(p)\leq\alpha^\prime$ if and only if $w(p)\leq\alpha^\prime$,
$r_u(N,\alpha^\prime)=r_w(N,\alpha^\prime)=\kappa+h-1$.  Therefore,
$\dim(E_\alpha^\prime\cap F_N)=\kappa+h-1$ for $F_\bullet\in
X^\circ_u$.  Moreover, for $F_\bullet\in X^\circ_u$,
\begin{eqnarray*}
\dim(E_{\delta^\prime-1}+F_M)
& = & \delta^\prime-1+M-r_u(M,\delta^\prime-1) \\
& = & \delta^\prime-1+M-M+\alpha^\prime-\alpha \\
& = & \kappa-1.
\end{eqnarray*}
Therefore, the generic fiber over $X_u$ is isomorphic to
$\mathbb{P}^{h-1}$. 
\end{proof}

\appendix
\section{A Purely Pattern Avoidance Characterization \\ (by Sara Billey and Jonathan Weed)}

Let $\Sym_\infty$ be the union of $\Sym_n$ for all $n \geq 1$.  There
exists a partial order on $\Sym_\infty$ determined by pattern
embeddings; we say $v \prec w$ if there is a pattern embedding of $v$
into $w$.  If the embedding of $v$ into $w$ is given by the set of
indices $Z = \{i_1, \ldots, i_m\}$, then we write $fl_Z(w) = v$, i.e.,
that the ``flattened'' version of $w$ consisting only of the indices
in $Z$ is the permutation $v$. 

Consider the set $\KL{m} = \{ w \in \Sym_\infty \, | \, P_{id, w}(1)
\leq m\}$ for any positive integer $m$.  By \cite[Thm. 1]{BilBra}, we
know $\KL{m}$ is the complement of a lower order ideal in the poset of
pattern embeddings.  Therefore, $\KL{m}$ can be characterized by
pattern avoidance for every $m\geq 1$.  For example, $\KL{1}$ is the
set of permutations avoiding $4231$ and $3412$.  The following theorem
gives a minimal set of patterns characerizing $\KL{2}$.

\begin{theorem}\label{t:kl2} $\KL{2}$ is equal to the set of permutations avoiding the 66 patterns

\begin{xalignat}{4}
&4 5 1 2 3 && 3 4 5 1 2 && 5 3 4 1 2 && 5 2 3 4 1 \nonumber \\
&4 5 2 3 1 && 3 5 1 6 2 4 && 5 2 3 6 1 4 && 5 2 6 3 1 4 \nonumber \\
&6 2 4 1 5 3 && 5 2 4 6 1 3 && 4 6 2 5 1 3 && 5 2 6 4 1 3 \nonumber \\ 
&5 4 6 2 1 3 && 3 6 1 4 5 2 && 4 6 1 3 5 2 && 3 6 4 1 5 2 \nonumber \\
&4 6 3 1 5 2 && 5 3 6 1 4 2 && 4 6 5 1 3 2 && 4 2 6 3 5 1 \nonumber \\
&6 3 2 5 4 1 && 6 3 5 2 4 1 && 6 4 2 5 3 1 && 6 5 3 4 2 1 \nonumber \\
&3 6 1 2 7 4 5 && 6 2 3 1 7 4 5 && 6 2 4 1 7 3 5 && 3 4 1 6 7 2 5 \nonumber \\
&4 2 3 6 7 1 5 && 4 2 6 3 7 1 5 && 4 2 6 7 3 1 5 && 3 7 1 2 5 6 4 \nonumber \\
&7 2 3 1 5 6 4 && 3 7 1 5 2 6 4 && 3 7 5 1 2 6 4 && 7 5 2 3 1 6 4 \label{eqn:66pats}\\
&6 2 5 1 7 3 4 && 7 2 6 1 4 5 3 && 3 4 1 7 5 6 2 && 3 5 1 7 4 6 2 \nonumber \\
&4 5 1 7 3 6 2 && 4 2 3 7 5 6 1 && 5 3 4 7 2 6 1 && 4 2 7 5 6 3 1 \nonumber \\
&3 4 1 2 7 8 5 6 && 4 2 3 1 7 8 5 6 && 3 4 1 7 2 8 5 6 && 4 2 3 7 1 8 5 6 \nonumber \\
&4 2 7 3 1 8 5 6 && 3 5 1 2 7 8 4 6 && 5 2 3 1 7 8 4 6 && 5 2 4 1 7 8 3 6 \nonumber \\
&3 4 1 2 8 6 7 5 && 4 2 3 1 8 6 7 5 && 3 4 1 8 2 6 7 5 && 4 2 3 8 1 6 7 5 \nonumber \\
&4 2 8 3 1 6 7 5 && 3 4 1 8 6 2 7 5 && 4 2 3 8 6 1 7 5 && 4 2 8 6 3 1 7 5 \nonumber \\
&3 5 1 2 8 6 7 4 && 5 2 3 1 8 6 7 4 && 3 6 1 2 8 5 7 4 && 6 2 3 1 8 5 7 4 \nonumber \\
&5 2 4 1 8 6 7 3 && 6 2 5 1 8 4 7 3. \nonumber
\end{xalignat}

\label{thm:pats}
\end{theorem}

Given $w \in \Sym_n$, the irreducible components of the singular locus
of the Schubert variety $X_w$ are themselves Schubert varieties. The
set of permutations indexing these irreducible components is called
the \textit{maximal singular locus}, and is denoted by $\ms(w)$.  The
proof of Theorem~\ref{t:kl2} follows from the next two lemmas relating
the maximal singular locus of a Schubert variety with patterns.  
\pagebreak
\begin{lemma} \cite[Sec. 13]{BilWar}
\label{lem:num-sing-comps} Consider a set $Z$ such that $fl_Z(w) =
4231$ or $3412$. Then $Z$ corresponds to a unique element of $\ms(w)$
if and only if the pattern does not occur as the dotted part of one of
the following patterns:
\begin{center}
\begin{equation}
\begin{tabular}{l l l}
$\dot{3}\dot{5}4\dot{1}\dot{2}$ & $\dot{4}3\dot{5}\dot{1}\dot{2}$ & $\dot{4}\dot{5}\dot{1}3\dot{2}$ \\
$\dot{4}\dot{5}2\dot{1}\dot{3}$ & $\dot{5}\dot{2}3\dot{4}\dot{1}$ & $\dot{5}\dot{2}\dot{4}3\dot{1}$ \\
$\dot{5}3\dot{2}\dot{4}\dot{1}$ & $\dot{5}\dot{3}\dot{4}2\dot{1}$ & $\dot{5}4\dot{2}\dot{3}\dot{1}$ \\
$\dot{6}\dot{3}52\dot{4}\dot{1}$ & $\dot{5}\dot{6}34\dot{1}\dot{2}$ & $\dot{5}2\dot{6}4\dot{1}\dot{3}$ 
\\$\dot{4}\dot{6}3\dot{1}5\dot{2}$.
\end{tabular}
\label{eqn:useless}
\end{equation}
\end{center}
\end{lemma}

\begin{remark} In contrast to Theorem~\ref{thm:pats}, it is interesting to
note that the set $\MS{2} = \{w \in \Sym_\infty \, : \, |\ms(w)| \geq
2\}$ is not characterized by pattern avoidance. For instance, the
permutation $x = 4631725$ has a maximal singular locus of size 2, so
$x \in \MS{2}$. However $x \prec w = 47318625$, but $w$ has maximal
singular locus of size 1.  
\end{remark}

\begin{lemma}
If $|\ms(w)| \geq k$, then there exists a pattern $v \prec w$ with at
most $4k$ entries such that $|\ms(v)| \geq k$. \label{thm:me}
\end{lemma}
\begin{proof}
For each element $x_i$ of $\ms(w)$, let $Z_i$ be the indices of $w$ such that $fl_{Z_i}(w)$ is the $4231$ or $3412$ pattern corresponding to $x_i$, and let $Z$ be the union of $Z_1, Z_2, \ldots, Z_k$. Then $|Z| \leq 4k$, since each element of $\ms(w)$ adds at most 4 indices to $Z$, so $\fl_Z(w)$ has at most $4k$ entries. Let $v = \fl_Z(w)$.

Given $x_i \in \ms(w)$, $\fl_{Z_i}(w)$ is a $4231$ or $3412$ pattern which is not a subpattern of one of the dotted patterns in (\ref{eqn:useless}). Since $Z_i \subset Z$, $\fl_{Z_i}(v) = \fl_{Z_i}(\fl_{Z}(w)) = \fl_{Z_i}(w)$, so $Z_i$ is a $4231$ or $3412$ pattern in $v$ as well. Furthermore, it cannot be a subpattern of one of the patterns in (\ref{eqn:useless}), since then it would be a subpattern of that pattern in $w$ as well. Hence $Z_i$ corresponds to a unique element of $\ms(v)$. So $|\ms(v)| \geq |\ms(w)| \geq k$, as desired.
\end{proof}

\begin{proof}[Proof of Theorem~\ref{thm:pats}] By Theorem~\ref{thm:main}, $\KL{2}$ is the set of
permutations which have at most 1 elements in the maximal singular
locus and avoid
\begin{equation}
\{653421, 632541, 463152, 526413, 546213, 465132\}.  \label{eqn:6pats}
\end{equation}
If $w \not \in \KL{2}$, then either it contains a pattern in 
\eqref{eqn:6pats} or it has at least two elements in its maximal
singular locus. The patterns of \eqref{eqn:6pats} are in
\eqref{eqn:66pats}.  We claim that any $w \in \Sym_\infty$ with
$|\ms(w)| \geq 2$ contains a pattern in \eqref{eqn:66pats}.
Therefore, $w \not \in \KL{2}$ contains a pattern from
\eqref{eqn:66pats}.

To prove the claim, note by Lemma~\ref{thm:me} that there exists $v
\in \Sym_{\leq 8}$ such that $v \prec w$ and $|\ms(v)| \geq 2$. A
computer check establishes that \eqref{eqn:66pats} contains all
the minimal patterns in $\Sym_{\leq 8}$ not in $\KL{2}$, hence $w$
contains one such pattern. 

Conversely, if $w \in S_{\infty}$ contains a pattern $v$ in
\eqref{eqn:66pats}, then a computer verification shows that
$P_{id, v}(1) > 2$ so by \cite{BilBra}, $P_{id, w}(1) > 2$.
Hence, $w$ is not in $\KL{2}$.  Therefore, the patterns of
\eqref{eqn:66pats} characterize $\KL{2}$, as desired.
\end{proof}

The structure of $KL_{m}$ for $m \geq 1$ gives a pattern avoidance
``filtration'' on $S_{\infty}$.  This suggests the following
questions.   

\begin{enumerate}
\item Can $KL_{m}$ always be characterized by a finite number of patterns?

\item If so, can the minimal elements of the complement of $KL_{m}$ be
determined efficiently?

\item We know the maximal singular locus is efficient to calculate.  Can we
use information about $\ms (w)$ to give bounds for $P_{id,w}(1)$?
\end{enumerate}

The following conjecture has been tested through $S_{8}$.  

\begin{conjecture}  Let $w \in S_{n}$. 
\begin{enumerate}
\item If $P_{id, w}(1) \leq 3$ then $|\ms(w)| \leq 3$.
\item If $P_{id, w}(1) =3$ and  $|\ms(w)| =1$ then $P_{id,w}=1+q^{a}+q^{b}$. 
\item If $P_{id, w}(1) =3$ and  $|\ms(w)| =2$ then $P_{id,w}=1+q^{a}+q^{b}$.
\item If $P_{id, w}(1) =3$ and  $|\ms(w)| =3$ then $P_{id,w}=1+2q^{a}$.  
\end{enumerate}
\end{conjecture}

The following conjecture has been tested for $B_{5}$, $C_{5}$, and
$D_{5}$.

\begin{conjecture}
For other Weyl group types, $P_{id,w}(1)=2$ implies $|\ms (w)|=1$.
\end{conjecture}

\end{document}